\documentclass{amsart}
\usepackage{amsmath}
\usepackage{amsthm}
\usepackage{amssymb}
\usepackage{enumerate}
\newtheorem{theorem}{Theorem}[section]
\newtheorem{lemma}[theorem]{Lemma}
\newtheorem{proposition}[theorem]{Proposition}
\newtheorem{corollary}[theorem]{Corollary}

\newtheorem{remar}[theorem]{Remark}
\newtheorem{prob}{Open Problem}[section]
\theoremstyle{definition}

\newcommand{\QED}{{\unskip\nobreak\hfil\penalty50%
\hskip1em\hbox{}\nobreak\hfil $\Box$%
\parfillskip=0pt \finalhyphendemerits=0 \par\medskip\noindent}}
\newcommand{\bfind}[1]{\index{#1}{\bf #1}}

\newcommand{\n}{\par\noindent}

\newcommand{\sn}{\par\smallskip\noindent}

\newcommand{\bn}{\par\bigskip\noindent}
\newcommand{\pars}{\par\smallskip}
\newcommand{\parm}{\par\medskip}
\newcommand{\parb}{\par\bigskip}
\newcommand{\cal}{\mathcal}

\newcommand{\chara}{\mbox{\rm char}\,}
\newcommand{\trdeg}{\mbox{\rm trdeg}\,}
\newcommand{\Gal}{\mbox{\rm Gal}\,}
\newcommand{\rr}{\mbox{\rm rr}\,}

\newcommand{\Q}{\mathbb Q}
\newcommand{\N}{\mathbb N}
\newcommand{\Z}{\mathbb Z}
\newcommand{\F}{\mathbb F}
\newcommand{\fvkadresse}{\par\bigskip \small\rm
 Department of Mathematics and Statistics, 
 University of Saskatchewan, \par
 106 Wiggins Road, 
 Saskatoon, Saskatchewan, Canada S7N 5E6 \par
 email: fvk@math.usask.ca}
 \newcommand{\abaddress}{\par\bigskip \small\rm
Institute of Mathematics,
University of Silesia, \par
Bankowa 14,
40-007 Katowice,
Poland \par
e-mail: ablaszczok@math.us.edu.pl}
\begin{document}
\title[Algebraic independence]{Algebraic independence of elements in
immediate extensions of valued fields}
\author{Anna Blaszczok and Franz-Viktor Kuhlmann}
\date{4.\ 4.\ 2013}
\begin{abstract}\noindent
Refining a constructive combinatorial method due to MacLane and
Schilling, we give several criteria for a valued field that guarantee
that all of its maximal immediate extensions have infinite transcendence
degree. If the value group of the field has countable cofinality, then
these criteria give the same information for the completions of the
field. The criteria have applications to the classification
of valuations on rational function fields. We also apply the criteria to
the question which extensions of a maximal valued field, algebraic or of
finite transcendence degree, are again maximal. In the case of valued
fields of infinite $p$-degree, we obtain the worst possible examples of
nonuniqueness of maximal immediate extensions: fields which admit an
algebraic maximal immediate extension as well as one of infinite
transcendence degree.
\end{abstract}
\thanks{During the work on this paper, the first author was partially
supported by a Polish UPGOW Project grant. The research of the second
author was partially supported by a Canadian NSERC grant and a
sabbatical grant from the University of Saskatchewan. He gratefully
acknowledges the hospitality of the Institute of Mathematics at the
Silesian University in Katowice.}
\maketitle
%
%
%
%
\section{Introduction}
In this paper, we denote a valued field by $(K,v)$, its value group by
$vK$, and its residue field by $Kv$. When we talk of a valued field
extension $(L|K,v)$ we mean that $(L,v)$ is a valued field, $L|K$ a
field extension, and $K$ is endowed with the restriction of $v$.
For the basic facts about valued fields, we refer the reader to
\cite{[E],[EP],[K5],[R],[W2],[ZS]}.

One of the basic problems in valuation theory is the description of the
possible extensions of a valuation from a valued field $(K,v)$ to a
given extension field $L$. The case of an algebraic extension $L|K$ is
taken care of by ramification theory.

Another important case is given when $L|K$ is an algebraic function
field. Valuations on algebraic function fields appear naturally in
algebraic geometry, real algebraic geometry and the model theory of
valued fields, to name only a few areas. Local uniformization, the local
form of resolution of singularities, is essentially a property of valued
algebraic function fields (cf.\ \cite{[K3]}). This problem, which is
still open in positive characteristic, requires a detailed knowledge of
all possible valuations on such function fields. The same is true for
corresponding problems in the model theory of valued fields.

By means of ramification theory, the problem of describing all appearing
valuations is reduced to the case of rational function fields. The case
of a single variable attracted many authors; see the references in
\cite{[K4]} for a selection from the extensive literature on this case.
The case of higher transcendence degree was treated in \cite{[K4]}. What
at first glance appeared to be problem easily solvable by induction,
turned out to tightly connected with the intricate question whether the
maximal immediate extensions of a given valued field have finite or
infinite transcendence degree.

An extension $(L|K,v)$ of valued fields is called \bfind{immediate} if
the canonical embeddings $vK\hookrightarrow vL$ and $Kv\hookrightarrow
Lv$ are onto. It was shown by W.~Krull \cite{[KR]} that maximal
immediate extensions exist for every valued field. The proof uses Zorn's
Lemma in combination with an upper bound for the cardinality of valued
fields with prescribed value group and residue field. Krull's deduction
of this upper bound is hard to read; later, K.~A.~H.~Gravett
\cite{[GRA]} gave a nice and simple proof.

A valued field $(K,v)$ is called \bfind{henselian} if it satisfies
Hensel's Lemma, or equivalently, if the extension of its valuation $v$
to its algebraic closure, which we will denote by $\tilde{K}$, is
unique. A \bfind{henselization} of $(K,v)$ is an extension which is
henselian and minimal in the sense that it can be embedded over $K$, as
a valued field, in every other henselian extension field of $(K,v)$.
Therefore, a henselization of $(K,v)$ can be found in every henselian
extension field, and henselizations are unique up to valuation
preserving isomorphism over $K$ (this is why we will speak of {\it the}
henselization of $(K,v)$. The henselization is an immediate
separable-algebraic extension.

A valued field is called \bfind{maximal} if it does not admit any proper
immediate extension; clearly, maximal immediate extensions are maximal
fields. I.~Kaplansky (\cite{[Ka]}) characterized the maximal field as
those in which every pseudo Cauchy sequence admits a pseudo limit. From
this result it follows that power series fields are maximal fields. For
example, for any field $k$ the Laurent series field $k((t))$ with the
$t$-adic valuation is a maximal immediate extension of $k(t)$, and it is
well known that $k((t))$ is of infinite transcendence degree over $k(t)$.
This can be shown by a cardinality argument (and some facts about field
extensions in case $k$ is not countable). A constructive proof
was given by MacLane and Schilling in Section 3 of \cite{[MS]}. Our main
theorem is a far-reaching generalization of their result. A part of this
theorem has already been applied in \cite{[K4]} to the problem described
above.


\begin{theorem}                             \label{MTai}
Take a valued field extension $(L|K,v)$ of finite transcendence
degree $\geq 0$, with $v$ nontrivial on $L$. Assume that one of the
following four cases holds:
\sn
\underline{valuation-transcendental case}: \
$vL/vK$ is not a torsion group, or $Lv|Kv$ is transcendental;
\sn
\underline{value-algebraic case}: \
$vL/vK$ contains elements of arbitrarily high order, or there is a
subgroup $\Gamma\subseteq vL$ containing $vK$ such that $\Gamma/vK$ is
an infinite torsion group and the order of each of its elements is prime
to the characteristic exponent of $Kv$;
\sn
\underline{residue-algebraic case}: \
$Lv$ contains elements of arbitrarily high degree over $Kv$;
\sn
\underline{separable-algebraic case}: \
$L|K$ contains a separable-algebraic subextension $L_0|K$
such that within some henselization of $L$, the corresponding
extension $L_0^h|K^h$ is infinite.
\sn
Then each maximal immediate extension of $(L,v)$ has infinite
transcendence degree over $L$. If the cofinality of $vL$ is countable
(which for instance is the case if $vL$ contains an element $\gamma$
such that $\gamma>vK$), then already the completion of $(L,v)$ has
infinite transcendence degree over $L$.
\end{theorem}

Note that the cofinality of $vL$ is equal to the cofinality of
$vK$ unless there is some $\gamma$ in $vL$ which is larger than every
element of $vK$. In that case, because $L|K$ has finite transcendence
degree, $vL$ will have countable cofinality, no matter what the
cofinality of $vK$ is.

Note further that the condition in the residue-algebraic case always
holds when $Lv|Kv$ contains an infinite separable-algebraic
subextension; this is a consequence of the Theorem of the Primitive
Element. There is no analogue of this theorem in abelian groups;
therefore, the first condition in the value-algebraic case does not
follow from the second. As an example, take $q$ to be a prime different
from $\chara Kv$ and consider the case where $vL/vK$ is an infinite
product of $\Z/q\Z$. Under the second condition, however, the result can
easily be deduced from the separable-algebraic case by passing to a
henselization $L^h$ of $L$ and using Hensel's Lemma to show that $L^h|K$
admits the required subextension. For the details, see the proof of
Theorem~\ref{MTai} in Section~\ref{secttrmi}.

The key assumption in the separable-algebraic case is that the
separable-algebraic subextension remains infinite when passing to the
respective henselizations. We show that this condition is crucial.
Take a valued field $(k,v)$ which has a transcendental maximal immediate
extension $(M,v)$. We know that $(M,v)$ is henselian
(cf.~Lemma~\ref{mfhdl}). Take a transcendence basis $\mathcal{T}$ of
$M|k$ and set $K:= k(\mathcal{T})$. Then from Lemma \ref{infH} it
follows that the henselization $K^h$ of $K$ inside of $(M,v)$ is an
infinite separable-algebraic subextension of $(M|K,v)$. But $M$ is a
maximal immediate extension of $L:=K^h$ and $M|L$ is algebraic. Hence
the assertion of Theorem \ref{MTai} does not necessarily hold without
the condition that $L_0^h|K^h$ is infinite.

\pars
An interesting special case is given when $(K,v)$ is itself a maximal
field. In this case, it is well known that if $(L|K,v)$ is a finite
extension, then $(L,v)$ is itself a maximal field. So we ask what
happens if $(L|K,v)$ is infinite algebraic or transcendental of finite
transcendence degree. Under which conditions could $(L,v)$ be again a
maximal field? This question will be addressed in Section~\ref{sectemf},
where all of the following theorems will be proved.

\begin{theorem}                           \label{MT2}
Take a maximal field $(K,v)$ and an infinite algebraic extension
$(L|K,v)$. Assume that $L|K$ contains an infinite separable subextension
or that
\begin{equation}                            \label{finpdeg}
(vK:pvK)[Kv:Kv^p]\><\>\infty\>,
\end{equation}
where $p$ is the characteristic exponent of $Kv$. Then every maximal
immediate extension of $(L,v)$ has infinite transcendence degree over
$L$.
\end{theorem}

As an immediate consequence, we obtain:

\begin{corollary}
Take a maximal field $(K,v)$ of characteristic $0$ and an algebraic
extension $(L|K,v)$. Then $(L,v)$ is maximal if and only if $L|K$ is a
finite extension.
\end{corollary}

It remains to discuss the case where $L|K$ is an infinite extension,
its maximal separable subextension $K'|K$ is finite, and condition
(\ref{finpdeg}) fails. Since then also $(K',v)$ is maximal, we can
replace $K$ by $K'$ and simply concentrate on the case where $L|K$ is
purely inseparable.

Note that if the maximal field $K$ is of characteristic $p$,
then condition (\ref{finpdeg}) implies that the $p$-degree of $K$ is
finite, as it is equal to $(vK:pvK)[Kv:Kv^p]$. If condition
(\ref{finpdeg}) does not hold, then the purely inseparable extension
$K^{1/p}|K$ is infinite; since $vK^{1/p}=\frac{1}{p}vK$ and $K^{1/p}v
=(Kv)^{1/p}$, we then have that $vK^{1/p}/vK$ is of exponent $p$, every
element in $K^{1/p}v\setminus Kv$ has degree $p$ over $Kv$, and at least
one of the two extensions is infinite. Since $(K^{1/p},v)$ is again
maximal (regardless of the $p$-degree of $K$, see Lemma~\ref{1/pmax}),
this case shows that the assertion of Theorem~\ref{MTai} may fail even
when $vL/vK$ is an infinite torsion group or $Lv|Kv$ is an infinite
algebraic extension.
In fact, all possible cases can appear for infinite $p$-degree:

\begin{theorem}                             \label{infpdeg}
Take a maximal field $(K,v)$ of characteristic $p>0$ for which condition
(\ref{finpdeg}) fails (which is equivalent to $K$ having infinite
$p$-degree). Take $\kappa$ to be the maximum of $(vK:pvK)$ and
$[Kv:Kv^p]$, considered as cardinals. Then:
\sn
a) \ The valued field $(K^{1/p},v)$ is again maximal, although
$vK^{1/p}/vK$ is an infinite torsion group or $K^{1/p}v|Kv$ is an
infinite algebraic extension.
\sn
b) \ For every $n\in\N$ and every infinite cardinal $\lambda\leq\kappa$,
there are subextensions $(L_n|K,v)$ and $(L_\lambda|K,v)$ of $(K^{1/p}
|K,v)$ such that $(K^{1/p}|L_\lambda,v)$ is an immediate algebraic
extension of degree $\lambda$ and $(K^{1/p}|L_n,v)$ is an
immediate algebraic extension of degree $p^n$.
\sn
c) \ There is a purely inseparable extension $(L|K,v)$ with
\sn
$\bullet$ \ $vL=\frac{1}{p}vK$ and $Lv=Kv$ if $(vK:pvK)=\infty$,
\sn
$\bullet$ \ $vL=vK$ and $Lv=(Kv)^{1/p}$ if $[Kv:Kv^p]=\infty$,
\sn
such that every maximal immediate extension of $(L,v)$
has transcendence degree at least $\kappa$. In both cases, $L$ can also
be taken to simultaneously satisfy $vL=\frac{1}{p}vK$ and
$Lv=(Kv)^{1/p}$.
\pars
If the cofinality of $vK$ is countable, then in b), $K^{1/p}$ can be
replaced by the completion of $L_\lambda$ or $L_n$, respectively, and in
c), ``maximal immediate extension'' can be replaced by ``completion''.
\end{theorem}
\n
Case b) of this theorem is a generalization of Nagata's example
(\cite[Appendix, Example~(E3.1), pp.~206-207]{[Na]}). Similar to that
example, the valued fields in b) are nonmaximal fields admitting an
algebraic maximal immediate extension. We note that the field $L$ of
part c) is \emph{not} contained in $K^{1/p}$.


It was shown by Kaplansky that if a valued field satisfies ``hypothesis
A'', then its maximal immediate extensions are unique up to isomorphism
(\cite[Theorem~5]{[Ka]}; see also \cite{[KPR]}). Kaplansky also gives an
example for a valued field for which uniqueness fails (\cite[Section
5]{[Ka]}). The question whether uniqueness always fails when hypothesis
A is violated is open. Different partial answers were given in
\cite{[KPR]} and in \cite{[W1]}. To the best knowledge of the authors,
the next theorem presents, for the first time in the literature, the
worst case of nonuniqueness:

\begin{theorem}                             \label{distinct_max}
Take a maximal field $(K,v)$ of characteristic $p>0$ satisfying one of
the following conditions:
\sn
i) \ $vK/pvK$ is infinite and $vK$ admits a set of representatives of
the cosets modulo $pvK$ which contains an infinite bounded subset, or
\sn
ii) \ the residue field extension $Kv|(Kv)^p$ is infinite and the value
group $vK$ is not discrete.
\sn
Then there is an infinite purely inseparable extension $(L,v)$ of
$(K,v)$ which admits $(K^{1/p},v)$ as an algebraic maximal immediate
extension, but also admits a maximal immediate extension of infinite
transcendence degree.
\end{theorem}
\n
Let us mention that the condition of i) always holds when $vK/pvK$ is
infinite and $vK$ is of finite rank, or $\Gamma/p\Gamma$ is infinite for
some archimedean component $\Gamma$ of $vK$. For example, if $\F_p$
denotes the field with $p$ elements and $G$ is an ordered subgroup of
the reals of the form $\bigoplus_{i\in\N} r_i\Z$, then the power series
field $\F_p((G))$ satisfies the condition of i). If $t_i\,$, $i\in\N$,
are algebraically independent over $\F_p\,$, then the power series field
$\F_p(t_i\mid i\in\N)((\Q))$ with residue field $\F_p(t_i\mid i\in\N)$
satisfies the condition of ii).

\parm
Finally, let us discuss the case of transcendental extensions $(L,v)$ of
a maximal field $(K,v)$. In view of the valuation-transcendental case of
Theorem~\ref{MTai}, it remains to consider the
\bfind{valuation-algebraic case} where $vL/vK$ is a torsion group and
$Lv|Kv$ is algebraic. Here is a partial answer to our above question:

\begin{theorem}                             \label{maxtrext}
Take a maximal field $(K,v)$ and a transcendental extension $(L,v)$ of
$(K,v)$ of finite transcendence degree. Assume that $Lv|Kv$ is
separable-algebraic and $vL/vK$ is a torsion group such that the
characteristic of $Kv$ does not divide the orders of its elements. Then
$Lv|Kv$ or $vL/vK$ is infinite and every maximal immediate extension of
$(L,v)$ has infinite transcendence degree over $L$.
\end{theorem}
\sn
We do not know an answer in the case where the conditions on the value
group and residue field extensions fail.


%
%
\section{Preliminaries}
By $va$ we denote the value of an element $a\in K$, and by $av$ its
residue. Given any subset $S$ of $K$, we define
\[
vS\>=\>\{va\mid 0\neq a\in S\} \; \mbox{ and } \; Sv\>=\>
\{av\mid a\in S, va\geq 0\}\>.
\]
The valuation ring of $(K,v)$ will be denoted by ${\cal O}_K\,$.

%
%
\subsection{The fundamental inequality}
Every finite extension $(L|K,v)$ of valued fields
satisfies the \bfind{fundamental inequality} (cf.\ (17.5)
of \cite{[E]} or Theorem 19 on p.~55 of \cite{[ZS]}):
\begin{equation}                             \label{fundineq}
[L:K]\>\geq\>(vL:vK)[Lv:Kv]\>.
\end{equation}
The nature of the ``missing factor'' on the right hand side is
determined by the \bfind{Lemma of Ostrowski} which says that whenever
the extension of $v$ from $K$ to $L$ is unique, then
\begin{equation}                            
[L:K]\;=\; p^\nu \cdot (vL:vK)\cdot [Lv:Kv] \;\;\;\mbox{ with }
\nu\geq 0\>,
\end{equation}
where $p$ is the \textbf{characteristic exponent} of $Lv$, that is,
$p=\chara Lv$ if this is positive, and $p=1$ otherwise. For the proof,
see \cite[Theoreme 2, p.~236]{[R]}) or \cite[Corollary to Theorem~25,
p.~78]{[ZS]}).

The factor ${\rm d}={\rm d}(L|K,v)=p^\nu$ is called the \bfind{defect}
of the
extension $(L|K,v)$. If d$\>=1$, then we call $(L|K,v)$ a
\bfind{defectless extension}.
Note that $(L|K,v)$ is always defectless if $\chara Kv=0$.

We call a henselian field $(K,v)$ a \bfind{defectless field} if equality
holds in the fundamental inequality (\ref{fundineq}) for every finite
extension $L$ of $K$.


\begin{theorem}                                    \label{mfhdl}
Every maximal field is henselian and a defectless field.
\end{theorem}
\begin{proof}
The henselization of a valued field is an immediate extension.
Therefore, a maximal field is equal to its henselization and thus
henselian. For a proof that maximal fields are defectless fields, see
\cite[Theorem~31.21]{[W2]}.
\end{proof}

%
%
\subsection{Some facts about henselian fields and henselizations}
Let $(K,v)$ be any valued field. If $a\in \tilde{K}\setminus K$ is not
purely inseparable over $K$, we choose some extension of $v$
from $K$ to $\tilde{K}$ and define
\[
\mbox{\rm kras}(a,K)\>:=\> \max\{v(\tau a-\sigma a)\mid
\sigma,\tau\in \Gal (\tilde{K}|K) \mbox{\ \ and\ \ } \tau a\ne \sigma a\}
\;\in\>v\tilde{K}
\]
and call it the \bfind{Krasner constant of $a$ over $K$}. Since all
extensions of $v$ from $K$ to $\tilde{K}$ are conjugate, this does not
depend on the choice of the particular extension of $v$. For the
same reason, over a henselian field $(K,v)$ our Krasner constant
$\mbox{\rm kras}(a,K)$ is equal to
\[
\max\{v(a-\sigma a)\mid \sigma\in \Gal (\tilde{K}|K)
\mbox{\ \ and\ \ } a\ne \sigma a\}\>.
\]

\begin{lemma}                               \label{vf(a)}
Take an extension $(K(a)|K,v)$ of henselian fields, where $a$ is an
element in the separable-algebraic closure of $K$ with $va\geq 0$. Then
\begin{equation}                            \label{krasa}
va\>\leq\>\mbox{\rm kras}(a,K)\>,
\end{equation}
and for every polynomial $f=d_mX^m+\ldots+d_0\in K[X]$ of degree
$m<[K(a):K]$,
\begin{equation}                            \label{krasfa}
vf(a)\>\leq\>vd_m\,+\,m\, \mbox{\rm kras}(a,K)\;.
\end{equation}
\end{lemma}
\begin{proof}
Since $(K,v)$ is henselian, $v\sigma a=a$ and therefore $v(a-\sigma
a)\geq va$ for all $\sigma$. This yields inequality (\ref{krasa}).

Take any element $b$ in the separable-algebraic closure of $K$ with
$[K(b):K]$ $<[K(a):K]$. Then $v(a-b)\leq\mbox{\rm kras}(a,K)$ since
otherwise, Krasner's Lemma would yield that $a\in K(b)$ and $[K(b):K]
\geq [K(a):K]$. If we write $f(X)=d_m\prod_{i=1}^m (X-b_i)$, then
$[K(b_i):K]\leq\deg (f)<[K(a):K]$. Hence,
\[
vf(a)\>=\>vd_m\,+\,\sum_{i=1}^{m} v(a-b_i)\>\leq\>vd_m\,+\,m\,
\mbox{\rm kras}(a,K)\;.
\]
This proves inequality (\ref{krasfa}).
\end{proof}

\begin{lemma}                                      \label{infH}
Take a nontrivially valued field $(k(\mathcal{T}),v)$, where
$\mathcal{T}$ is a nonempty set of elements algebraically independent
over $k$. Then the henselization of $(k(\mathcal{T}),v)$ inside of any
henselian valued extension field is an infinite extension of
$k(\mathcal{T})$.
\end{lemma}
\begin{proof}
Set $F:=k(\mathcal{T})$ and take a henselization $F^h$ of $F$ inside of
some henselian valued extension field. Pick an arbitrary $t\in
\mathcal{T}$. Without loss of generality we can assume that $vt>0$. By
Hensel's Lemma, $F^h$ contains a root $\vartheta_1$ of the polynomial
$X^2-X-t$ such that $v\vartheta_1>0$. We proceed by induction. Once we
have constructed $\vartheta_i$ with $v\vartheta_i>0$ for some $i\in\N$,
we again use Hensel's Lemma to obtain a root $\vartheta_{i+1}\in F^h$ of
the polynomial $X^2-X-\vartheta_i$ with $v\vartheta_{i+1}>0$.

It now suffices to show that the extension $F(\vartheta_i\,|\,i\in\N)|F$
is infinite. To this end, we consider the $t^{-1}$-adic valuation $w$ on
$F=k(\mathcal{T}\setminus \{t\})(t^{-1})$ which is trivial on $k(\mathcal{T}
\setminus \{t\})$. We note that $wF=\Z$. Since $wt<0$, we obtain that
$w\vartheta_1=\frac{1}{2} wt$ and by induction, $w\vartheta_i=
\frac{1}{2^i} wt$. Therefore, the $2$-divisible hull of $\Z$ is
contained in $wF(\vartheta_i\,|\,i\in\N)$. In view of the fundamental
inequality (\ref{fundineq}), this shows that $F(\vartheta_i\,|\,i\in\N)|
F$ cannot be a finite extension.
\end{proof}

\begin{lemma}                               \label{vrrac}
Assume $(L,v)$ to be henselian and $K$ to be relatively
separable-algebraically closed in $L$. Then $Kv$ is relatively
separable-algebraically closed in $Lv$. If in addition $Lv|Kv$ is
algebraic, then the torsion subgroup of $vL/vK$ is
a $p$-group, where $p$ is the characteristic exponent of $Kv$.
\end{lemma}
\begin{proof}
Take $\zeta\in Lv$ separable-algebraic over $Kv$. Choose a monic
polynomial $g(X)\in K[X]$ whose reduction $gv(X)\in Kv[X]$ modulo
$v$ is the minimal polynomial of $\zeta$ over $Kv$. Then $\zeta$ is a
simple root of $gv$. Hence by Hensel's Lemma, there is a root $a\in
L$ of $g$ whose residue is $\zeta$. As all roots of $gv$ are
distinct, we can lift them all to distinct roots of $g$. Thus, $a$ is
separable-algebraic over $K$. From the assumption of the lemma, it
follows that $a\in K$, showing that $\zeta\in Kv$. This proves that $Kv$
is relatively separable-algebraically closed in $Lv$.

Now assume in addition that $Lv|Kv$ is algebraic. Then $Kv$ is
relatively separable-algebraically closed in $Lv$, by what we have
proved already. Take $\alpha\in vL$ and $n\in \N$ not divisible by $p$
such that $n\alpha\in vK$. Choose $a\in L$ and $b\in K$ such that
$va=\alpha$ and $vb=n\alpha$. Then $v(a^n/b) = 0$. Since $Lv|Kv$ is a
purely inseparable extension, there exists $m\in\N$ such that
$((a^n/b)v)^{p^m}\in Kv$. We choose $c\in K$ satisfying $vc=0$ and
$cv=((a^n/b)v)^{p^m}$, to obtain that $(a^{np^m}/ cb^{p^m})v=1$. So the
reduction of the polynomial $X^n-a^{np^m}/ cb^{p^m}$ modulo $v$ is
$X^n-1$. Since $n$ is not divisible by $p$, $1$ is a simple root of this
polynomial. Hence by Hensel's Lemma, there is a simple root $d\in L$ of
the polynomial $X^n-a^{np^m}/cb^{p^m}$ with $dv=1$, whence $vd=0$.
Consequently, $a^{p^m}/d$ is a simple root of the polynomial
$X^n-cb^{p^m}$ and thus is separable algebraic over $K$. Since $K$ was
assumed to be relatively separable-algebraically closed in
$L$, we find that $a^{p^m}/d\in K$. As $n$ is not divisible by $p$,
there are $k,l\in\Z$ such that $1=kn+lp^m$. This yields:
\[
\alpha\>=\> kn\alpha+lp^m\alpha\>=\> kn\alpha+l(p^m va - vd) \>=\>
k(n\alpha)+lv\left(\frac{a^{p^m}}{d}\right) \in vK.
\]
\end{proof}

%
%
\subsection{Valuation independence}
For the easy proof of the following lemma, see \cite[chapter VI,
\S10.3, Theorem~1]{[B]}
.
\begin{lemma}                                      \label{prelBour}
Let $(L|K,v)$ be an extension of valued fields. Take elements $x_i,y_j
\in L$, $i\in I$, $j\in J$, such that the values $vx_i\,$, $i\in I$,
are rationally independent over $vK$, and the residues $y_jv$, $j\in
J$, are algebraically independent over $Kv$. Then the elements
$x_i,y_j$, $i\in I$, $j\in J$, are algebraically independent over $K$.

Moreover, if we write
\[
f\>=\> \displaystyle\sum_{k}^{} c_{k}\,
\prod_{i\in I}^{} x_i^{\mu_{k,i}} \prod_{j\in J}^{} y_j^{\nu_{k,j}}\in
K[x_i,y_j\mid i\in I,j\in J]
\]
in such a way that for every $k\ne\ell$
there is some $i$ s.t.\ $\mu_{k,i}\ne\mu_{\ell,i}$ or some $j$ s.t.\
$\nu_{k,j}\ne\nu_{\ell,j}\,$, then
\begin{equation}                            \label{value}
vf\>=\>\min_k\, v\,c_k \prod_{i\in I}^{}
x_i^{\mu_{k,i}}\prod_{j\in J}^{} y_j^{\nu_{k,j}}\>=\>
\min_k\, vc_k\,+\,\sum_{i\in I}^{} \mu_{k,i} v x_i\;.
\end{equation}
That is, the value of the polynomial $f$ is equal to the least of the
values of its monomials. In particular, this implies:
\begin{eqnarray*}
vK(x_i,y_j\mid i\in I,j\in J) & = & vK\oplus\bigoplus_{i\in I}
\Z vx_i\\
K(x_i,y_j\mid i\in I,j\in J)v & = & Kv\,(y_jv\mid j\in J)\;.
\end{eqnarray*}
Moreover, the valuation $v$ on $K(x_i,y_j\mid i\in I,j\in J)$ is
uniquely determined by its restriction to $K$, the values $vx_i$ and
the residues $y_jv$.
\parm
Conversely, if $(K,v)$ is any valued field and we assign to the $vx_i$
any values in an ordered group extension of $vK$ which are rationally
independent, then (\ref{value}) defines a valuation on $L$, and the
residues $y_jv$, $j\in J$, are algebraically independent over $Kv$.
\end{lemma}

\begin{corollary}                              \label{fingentb}
Let $(L|K,v)$ be an extension of valued fields. Then
\begin{equation}                            \label{wtdgeq}
\trdeg L|K \>\geq\> \trdeg Lv|Kv \,+\, \rr (vL/vK)\;.
\end{equation}
If in addition $L|K$ is a function field and if equality holds in
(\ref{wtdgeq}), then the extensions $vL| vK$ and $Lv|Kv$ are finitely
generated.
\end{corollary}
\begin{proof}
Choose elements $x_1,\ldots,x_{\rho},y_1,\ldots,y_{\tau}\in L$ such that
the values $vx_1,\ldots,vx_{\rho}$ are rationally independent over $vK$
and the residues $y_1v,\ldots,y_{\tau} v$ are algebraically independent
over $Kv$. Then by the foregoing lemma, $\rho+\tau\leq\trdeg L|K$. This
proves that $\trdeg Lv|Kv$ and the rational rank of $vL/vK$ are finite.
Therefore, we may choose the elements $x_i,y_j$ such that $\tau=\trdeg
Lv|Kv$ and $\rho=\dim_{\Q} \Q\otimes (vL/vK)$ to obtain inequality
(\ref{wtdgeq}).

Set $L_0:=K(x_1,\ldots, x_{\rho},y_1,\ldots,y_{\tau})$ and
assume that equality holds in (\ref{wtdgeq}). This means that the
extension $L|L_0$ is algebraic. Since $L|K$ is finitely generated, it
follows that this extension is finite. By the fundamental inequality
(\ref{fundineq}), this yields that $vL|vL_0$ and $Lv| L_0v$ are finite
extensions. Since already $v L_0|v K$ and $L_0v|Kv$ are finitely
generated by the foregoing lemma, it follows that also $vL|vK$ and
$Lv|Kv$ are finitely generated.
\end{proof}

\pars
The algebraic analogue to the transcendental case discussed in
Lemma~\ref{prelBour} is the following lemma (see \cite{[ZS]} for a
proof):
\begin{lemma}                               \label{algvind}
Let $(L|K,v)$ be an extension of valued fields. Suppose that $\eta_1,
\ldots,\eta_k\in L$ such that $v\eta_1,\ldots,v\eta_k\in vL$ belong to
distinct cosets modulo $vK$. Further, assume that $\vartheta_1,\ldots,
\vartheta_{\ell}\in {\cal O}_L$ such that $\vartheta_1v,\ldots,
\vartheta_{\ell}v$ are $Kv$-linearly independent. Then the elements
$\eta_i\vartheta_j\,$, $1\leq i\leq k$, $1\leq j\leq\ell$, are
$K$-linearly independent, and for every choice of elements $c_{ij}
\in K$, we have that
\[
v\sum_{i,j} c_{ij} \eta_i\vartheta_j\>=\> \min_{i,j}\, v c_{ij}
\eta_i\vartheta_j\>=\> \min_{i,j}\, (vc_{ij}\,+\,v\eta_i)\>.
\]
If the elements $\eta_i\vartheta_j$ form a $K$-basis of $L$, then
\[
vL\>=\>vK+\bigoplus_{1\leq i\leq k} \Z v\eta_i\;\;\mbox{ \ \
and \ \ }\;\; Lv\>=\>Kv(\vartheta_j v\mid 1\leq j\leq\ell)\>.
\]
\end{lemma}

For any element $x$ in a field extension of $K$ and every nonnegative
integer $n$, we set
\[
K[x]_n\>:=\> K+Kx+\ldots+Kx^n \>.
\]
Since $\dim_K K[x]_n\leq n+1$, we obtain the following corollary
from Lemma~\ref{algvind}:

\begin{corollary}                              \label{polvspaces}
Take a valued field extension $(K(x)|K,v)$. Then for every $n\geq 0$,
\sn
a) \ the elements of $vK[x]_n$ lie in at most $n+1$ many distinct cosets
modulo $vK$,
\sn
b) \ the $Kv$-vector space $K[x]_n v$ is of dimension at most $n+1$.
\end{corollary}

%
%
\subsection{Immediate extensions}
We will assume some familiarity with the basic properties of pseudo
Cauchy sequences; we refer the reader to Kaplansky's paper ``Maximal
fields with valuations'' (\cite{[Ka]}). In particular, we will use the
following two main theorems:

\begin{theorem} {\rm\ \ \ (Theorem 2 of \cite{[Ka]})}\label{KT2}
\n
For every pseudo Cauchy sequence $(a_\nu)_{\nu<\lambda}$ in $(K,v)$ of
transcendental type there exists an immediate transcendental
extension $(K(x),v)$ such that $x$ is a pseudo limit of
$(a_\nu)_{\nu<\lambda}\,$. If $(K(y),v)$ is another valued extension
field of $(K,v)$ such that $y$ is a pseudo limit of
$(a_\nu)_{\nu<\lambda}\,$, then $y$ is also transcendental over $K$ and
the isomorphism between $K(x)$ and $K(y)$ over $K$ sending $x$ to $y$ is
valuation preserving.
\end{theorem}

\begin{theorem} {\rm\ \ \ (Theorem 3 of \cite{[Ka]})}\label{KT3}
\n
Take a pseudo Cauchy sequence $(a_\nu)_{\nu<\lambda}$ in $(K,v)$ of
algebraic type. Choose a polynomial $f(X)\in K[X]$ of minimal degree
whose value is not fixed by $(a_\nu)_{\nu<\lambda}\,$, and a root $z$ of
$f$. Then there exists an extension of $v$ from $K$ to $K(z)$ such that
$(K(z)|K,v)$ is an immediate extension and $z$ is a pseudo limit of
$(a_\nu)_{\nu<\lambda}\,$.

If $(K(z'),v)$ is another valued extension field of $(K,v)$ such that
$z'$ is also a root of $f$ and a pseudo limit of
$(a_\nu)_{\nu<\lambda}\,$, then the field isomorphism between $K(a)$ and
$K(b)$ over $K$ sending $a$ to $b$ will preserve the valuation.
\end{theorem}

We will need a few more results that are not in Kaplansky's paper.

\begin{lemma}
Take an algebraic algebraic field extension $(K(a)|K,v)$, where $a$ is
a pseudo limit of a pseudo Cauchy sequence $(a_{\nu})_{\nu<\lambda}$ in
$(K,v)$ without a pseudo limit in $K$. Then $(a_{\nu})_{\nu<\lambda}$
does not fix the value of the minimal polynomial of $a$ over $K$.
\end{lemma}
\begin{proof}
We denote the minimal polynomial of $a$ over $K$ by
$f(X)=\prod_{i=1}^n (X-\sigma_i a)$ with $\sigma_i\in \Gal(\tilde{K}
|K)$. Since $a$ is a pseudo limit of $(a_{\nu})_{\nu<\lambda}$, the
values $v(a_{\nu}-a)$ are ultimately increasing. If $v(a-\sigma_i a)>
v(a_{\nu}-a)$ for all $\nu<\lambda$, then also the values
$v(a_{\nu}-\sigma_i a)=\min\{v(a_{\nu}-a), v(a-\sigma_i
a)\}=v(a_{\nu}-a)$ are ultimately increasing. If on the other hand,
$v(a-\sigma_i a)\leq v(a_{\nu_0}-a)$ for some $\nu_0<\lambda$, then for
$\nu_0<\nu<\lambda$, the value $v(a_{\nu}-\sigma_i a)=\min \{v(a_{\nu}-a),
v(a-\sigma_i a)\}=v(a-\sigma_i a)$ is fixed. We conclude that the values
$vf(a_{\nu})=\sum_{i=1}^n v(a_{\nu}-\sigma_i a)$ are ultimately
increasing.
\end{proof}

\begin{lemma}                                    \label{pimm}
Take a henselian field $(K,v)$ of positive characteristic $p$ and a
pseudo Cauchy sequence $(a_{\nu})_{\nu<\lambda}$ in $(K,v)$ without a
pseudo limit in $K$. If $(K(a)|K,v)$ is a valued field extension of
degree $p$ such that $a$ is a pseudo limit of
$(a_{\nu})_{\nu<\lambda}\,$, then $(K(a)|K,v)$ is immediate.
\end{lemma}
\begin{proof}
By the previous lemma, $(a_{\nu})_{\nu<\lambda}$ does not fix the value of
the minimal polynomial $f$ of $a$ over $K$. On the other hand, we will
show that $(a_{\nu})_{\nu<\lambda}$ fixes the value of every polynomial of
degree less than $\deg f=p$. We take $g\in K[X]$ to be a polynomial of
smallest degree such that $(a_{\nu})_{\nu<\lambda}$ does not fix the
value of $g$. Since $(a_{\nu})_{\nu<\lambda}$ admits no pseudo limit in
$(K,v)$, the polynomial $g$ is of degree at least 2. Take a root $b$ of
$g$. By Theorem~\ref{KT3}, there is an extension of the valuation $v$
from $K$ to $K(b)$ such that $(K(b)|K,v)$ is immediate. Since $[K(b):K]
\geq 2$ and $(K,v)$ is henselian, the Lemma of Ostrowski implies that
$[K(b):K]\geq p$. This shows that $f$ is a polynomial of smallest degree
whose value is not fixed by $(a_{\nu})_{\nu<\lambda}$. Hence again by
Theorem~\ref{KT3}, there is an extension of the valuation $v$ from $K$
to $K(a)$ such that $(K(a)|K,v)$ is immediate. Since $(K,v)$ is
henselian, this extension coincides with the given valuation on $K(a)$
and we have thus proved that the extension $(K(a)|K,v)$ is immediate.
\end{proof}

The following result is Proposition 4.3 of \cite{[Ku6]}:
\begin{proposition}                                \label{deformthm}
Take a valued field $(F,v)$ of positive characteristic $p$. Assume
that $F$ admits an immediate purely inseparable extension $F(\eta)$ of
degree $p$ such that the element $\eta$ does not lie in the completion
of $(F,v)$. Then for each element $b\in F^{\times}$ such that
\begin{equation}                                   \label{dist}
(p-1)vb+v\eta\> >\> pv(\eta-c)
\end{equation}
holds for every $c\in F$, any root $\vartheta$ of the polynomial
\[
X^p-X-\left(\frac{\eta}{b}\right)^p
\]
generates an immediate Galois extension $(F(\vartheta)|F,v)$ of degree
$p$ with a unique extension of the valuation $v$ from $F$ to
$F(\vartheta)$.
\end{proposition}

%
%
\subsection{Characteristic blind Taylor expansion}
We need a Taylor expansion that works in all characteristics. For
polynomials $f\in K[X]$, we define the \bfind{$i$-th formal derivative
of $f$} as
\begin{equation}                            \label{iderivative}
f_i(X)\>:=\> \sum_{j=i}^{n}\binom{j}{i} c_j X^{j-i}
\>=\> \sum_{j=0}^{n-i}\binom{j+i}{i} c_{j+i} X^{j}\;.
\end{equation}
Then regardless of the characteristic of $K$, we have the
\bfind{Taylor expansion} of $f$ at $c$ in the following form:
\begin{equation}                          \label{Taylorexp}
f(X) \>=\> \sum_{i=0}^{n} f_i(c) (X-c)^i\;.
\end{equation}

%
%
%
\section{Algebraic independence of elements in maximal immediate
extensions}                                 \label{secttrmi}
This section is devoted to the proof of Theorem~\ref{MTai}.
Our first goal is a basic independence lemma.

Take $i\in\N$, any field $K$ and a polynomial $f\in K[X_1,\ldots,X_i]$.
With respect to the lexicographic order on $\Z^{i}$, let $(\mu_1,\ldots,
\mu_i)$ be maximal with the property that the coefficient of
$X_1^{\mu_1}\cdots X_i^{\mu_i}$ in $f$ is nonzero. Then define
$c_f$ to be this coefficient and call $(\mu_1,\ldots,\mu_i)$ the
\bfind{crucial exponent} of $f$.

\pars
For our basic independence lemma, we consider the following situation.
We choose a function
\[
\varphi:\;\N\times\N\;\longrightarrow\;\N
\]
such that
\[
\varphi(k,\ell)>\max\{k,\ell\}\;\mbox{\ \ and\ \ }\;\varphi(k+1,\ell)>
\varphi(k,\ell)\quad \mbox{for all $k,\ell\in\N$,}
\]
and for each $i\in\N$ a strictly increasing sequence $(E_i(k))_{k\in\N}$
of integers $\geq 2$ such that for all $k\geq 1$ and $i\geq 2$,
\begin{equation}                            \label{E}
\left.
\begin{array}{rcl}
E_1(k+1) &\geq& \varphi(k,E_1(k))+1,\\
E_i(k+1) &\geq& E_{i-1}(\varphi(k,E_i(k))+1)\;.
\end{array}
\right\}
\end{equation}
Then for $i,k\in\N$,
\begin{equation}                            \label{k+1k}
E_i(k)\> > \> k\;\mbox{\ \ and\ \ }\;
E_i(k+1)\> \geq \> \varphi(k,E_i(k))+1\> > \>E_i(k)+1\;.
\end{equation}

Further, we take an extension $(L|K,v)$ of valued fields, elements
\[
a_j\in L\;\mbox{\ \ and\ \ }\;\alpha_j\in vL\quad \mbox{for all
$j\in\N$,}
\]
and $K$-subspaces
\[
S_j\>\subseteq\>L\>, \quad j\in\N\>.
\]
We assume that for all $i,k,\ell\in\N$, the following conditions are
satisfied:
\sn
{\bf (A1)} \ $0\leq va_k\leq\alpha_k<va_{k+1}$ and $k\alpha_{E_i(k)}\leq
\alpha_{\varphi(k,E_i(k))}\,$,
\sn
{\bf (A2)} \ $a_1,\ldots,a_k\in S_k$ and $S_k\subseteq S_{k+1}\,$,
\sn
{\bf (A3)} \ if $d_0,\ldots,d_k\in S_k$ and $u\in S_{\ell}\,$, then
\[d_0+d_1u +\ldots+d_ku^k\>\in\> S_{\varphi(k,\ell)}\;,\]
\sn
{\bf (A4)} \ if $m\leq k$ and $d_0,\ldots,d_m\in S_k\,$, then
\[v(d_0+d_1a_{k+1}+\ldots+d_ma_{k+1}^m)\>\leq\>vd_m+m\alpha_{k+1}\;.\]

\pars
Now we choose any maximal immediate extension $(M,v)$ of $(L,v)$. For
each $i$, we take an arbitrary pseudo limit $y_i\in M$ of the pseudo
Cauchy sequence
\[
\left(\sum_{j=1}^{k} a_{E_i(j)}\right)_{k\in\N}\;.
\]
In this situation, we can prove the following basic independence lemma:

\begin{lemma}                               \label{btl}
Suppose that $k\geq 2$ is an integer and $f\in L[X_1,\ldots,X_i]$ is a
polynomial with coefficients in $S_{k-1}\cap {\cal O}_L$ such that
$\alpha_{E_i(k)}\geq vc_f$ and that $f$ has degree less than $k$ in each
variable. Then
\begin{equation}                            \label{ass}
vf(y_1,\ldots,y_i)\> < \> va_{E_i(k+1)}\>.
\end{equation}
\end{lemma}
\begin{proof}
We shall prove the lemma by induction on $i$. We start with $i=1$ and
set
\[
u\>:=\>\sum_{j=1}^{k} a_{E_1(j)}
\quad\mbox{\ \ and\ \ }\quad z\>:=\>y_1-u\;.
\]
Then $u\in S_{E_1(k)}$ because of (A2) and the fact that the $S_k$ are
vector spaces. By (A1), the definition of $E_1$ and our assumption that
$\alpha_{E_1(k)}\geq vc_f\,$,
\begin{equation}                            \label{vz>alg}
vz\>=\>va_{E_1(k+1)}\>\geq\>va_{\varphi(k,E_1(k))+1}\>>\>
\alpha_{\varphi(k,E_1(k))}\>\geq\> k\alpha_{E_1(k)}
\>\geq\>vc_f+(k-1)\alpha_{E_1(k)}\>.
\end{equation}
We use the Taylor expansion
\begin{equation}                            \label{vz>alg1}
f(y_1)\>=\>f(u+z)\>=\>f(u)\,+\,zf_1(u)\,+\,z^2f_2(u)\,+\,\ldots
\end{equation}
where $f_j(X)\in {\cal O}_L[X]$ is the $j$-th formal derivative of $f$
as defined in (\ref{iderivative}). We have that $f_j(u)\in {\cal O}_L$
for all $j$. Hence,
\begin{equation}                            \label{vTay>alg}
v(zf_1(u)\,+\,z^2f_2(u)\,+\,\ldots)\>\geq\>vz\;.
\end{equation}
We wish to prove that $vf(u)<vz$. We set
\[
u'\>:=\>\sum_{j=1}^{k-1} a_{E_1(j)}\>\in\>S_{E_1(k-1)}
\]
so that $u=u'+a_{E_1(k)}$. We use the Taylor expansion
\[
f(u)\>=\>f(u'+a_{E_1(k)})\>=\>f(u')\,+\,f_1(u')a_{E_1(k)}\,+\ldots+\,
f_m(u')a_{E_1(k)}^m
\]
where $m=\deg f< k$. By definition, $c_f$ is the leading coefficient of
$f$, which in turn is equal to the constant $f_m(u')= f_m(X)\in L$.
Since $f$ has coefficients in the vector space $S_{k-1}\,$, we know from
(\ref{iderivative}) that also all $f_j$ have coefficients in
$S_{k-1}\,$. Thus, (A3) and (A2) show that
\[
f(u'),\,f_j(u')\>\in\> S_{\varphi(k-1,E_1(k-1))}\>\subseteq\> S_{E_1(k)-1}
\]
for each $j$. Further, $m\leq k-1\leq E_1(k)-1$. Hence by (A4) and
(\ref{vz>alg}),
\[
vf(u)\>\leq\> vf_m(u') + m\alpha_{E_1(k)}\>\leq\>
vc_f\,+\,(k-1)\alpha_{E_1(k)}\><\>vz\;.
\]
From this together with (\ref{vz>alg1}) and (\ref{vTay>alg}), we deduce
that
\[
vf(y_1)\>=\>vf(u)\><\> vz\;,
\]
which gives the assertion of our lemma for the case of $i=1$.

\parb
In the case of $i>1$ we assume that the assertion of our lemma has
been proven for $i-1$ in place of $i$, and we set
\[
u\>:=\>\sum_{j=1}^{k} a_{E_i(j)}\;\in\>S_{E_i(k)}\>,\quad
u'\>:=\>\sum_{j=1}^{k-1} a_{E_i(j)}\>\in\>S_{E_i(k-1)}\>,
\quad\mbox{and}\quad z\>:=\>y_i-u\;.
\]
Then by (\ref{k+1k}), (A1) and our assumption that $\alpha_{E_i(k)}\geq
vc_f\,$,
\[
vz\>=\>va_{E_i(k+1)}\>\geq\>va_{\varphi(k,E_i(k))+1}
\> > \>\alpha_{\varphi(k,E_i(k))}\>\geq\>
k\alpha_{E_i(k)} \> \geq \>vc_f+ (k-1)\alpha_{E_i(k)}\>.
\]
We use the Taylor expansion
\begin{eqnarray*}
f(y_1,\ldots,y_{i-1},u+z) & = & f(y_1,\ldots,y_{i-1},u)\,+\,
zf_1(y_1,\ldots,y_{i-1},u) \\
 & + & z^2f_2(y_1,\ldots,y_{i-1},u)\,+\,\ldots
\end{eqnarray*}
where $f_j\in {\cal O}_L[X_1,\ldots,X_i]$ is the $j$-th formal
derivative of $f$ with respect to $X_i\,$. We obtain the analogue of
inequality (\ref{vTay>alg}); hence it will suffice to prove that
\begin{equation}
vf(y_1,\ldots,y_{i-1},u)\><\>vz\;.
\end{equation}
We set
%
%
\[
g(X_1,\ldots,X_{i-1})\>:=f\left(X_1,\ldots,X_{i-1},u\right)
\]
%
%
so that $g(y_1,\ldots, y_{i-1})=f(y_1,\ldots,y_{i-1},u)$. Viewing
$f$ as a polynomial in the variables $X_1,\ldots, X_{i-1}$ with
coefficients in $L[X_i]$, we denote by $h(X_i)$ the coefficient of
$X_1^{\mu_1}\cdots X_{i-1}^{\mu_{i-1}}$ in $f$.
Note that $h$ has coefficients in $S_{k-1}$, its leading
coefficient is $c_f$ and its degree is $\mu_i<k$. Again, since $h$ has
coefficients in $S_{k-1}\,$, (\ref{iderivative}) shows that the same is
true for the $j$-th formal derivative $h_j$ of $h$, for all $j$. Thus,
(A3), (\ref{k+1k}) and (A2) imply that
\[
h(u'),\, h_j(u')\>\in\> S_{\varphi(k-1,E_i(k-1))}\>\subseteq\>
S_{E_i(k)-1}
\]
for each $j$. As in the first part of our proof we find that
\[
vh(u)\>\leq\> vh_{\mu_i}(u') + \mu_i\alpha_{E_i(k)}\>=\>
vc_f\,+\,\mu_i \alpha_{E_i(k)}
\]
since $h_{\mu_i}(u')=c_f\,$. In particular, this shows that $h(u)\ne 0$.
Hence if $(\mu_1, \ldots,\mu_i)$ is the crucial exponent of $f$, then
$(\mu_1,\ldots,\mu_{i-1})$ is the crucial exponent of $g$, and
\[
c_g\>=\>h(u)\;.
\]
We set
\[
k'\>:=\>\varphi(k,E_i(k))\> >\>\max\{k,E_i(k)\}\;.
\]
Since $\mu_i\leq k-1$ and $vc_f\leq\alpha_{E_i(k)}$ by assumption, and
by virtue of (A1) and (\ref{k+1k}), it follows that
\[
vc_g\>=\>vh(u)\>\leq\> vc_f+(k-1)\alpha_{E_i(k)} \>\leq\> k\alpha_{E_i(k)}
\>\leq\>\alpha_{k'}\><\>\alpha_{E_{i-1}(k')}\>.
\]
%
%
Since every coefficient of $g$ is of the form $h(u)$ with $h$ a
polynomial of degree less than $k$ and coefficients in $S_{k-1}\,$, we
know from (A3), our conditions on $\varphi$ and (A2) that the
coefficients of $g$ lie in
\[
S_{\varphi(k-1,E_i(k))}\>\subseteq\>S_{\varphi(k,E_i(k))-1}
\>=\>S_{k'-1}\;.
\]
Also, its degree in each variable is less than $k$, hence less than
$k'$. Therefore, we can apply the induction hypothesis to the case of
$i-1$, with $k'$ in place of $k$. We obtain, by (A1) and our choice of
the numbers $E_i(k)$:
\[
vf(y_1,\ldots,y_{i-1},u) \><\> va_{E_{i-1}(\varphi(k,E_i(k))+1)}
\>\leq\> va_{E_i(k+1)}\>=\>vz\;.
\]
This establishes our lemma.
\end{proof}

By (A2), 
\[
S_{\infty}\>:=\>\bigcup_{k\in\N}^{} S_k
\]
contains $a_k$ for all $k$. We set
\[
K_{\infty}\>:=\>K(S_{\infty})\;.
\]
Further, we note that condition (A1) implies that
\[
\Gamma\>:=\>\{\alpha\in vK_{\infty}\mid -va_k\leq\alpha\leq va_k
\mbox{ for some $k$}\}
\]
is a convex subgroup of $vK_{\infty}\,$.

\begin{corollary}                           \label{btcor}
Assume that every element of $K_{\infty}$ with value in $\Gamma$ can be
written as a quotient $r/s$ with $r,s\in S_{\infty}$ such that $0\leq
vs\in\Gamma$. Then the elements $y_i\,$, $i\in\N$, are algebraically
independent over $K_{\infty}\,$.
\end{corollary}
\begin{proof}
We have to check that $g(y_1,\ldots,y_i)\ne 0$ for all $i$ and all
nonzero polynomials $g(X_1,\ldots,X_i)\in K_{\infty}[X_1,\ldots,X_i]$.
After division by some coefficient of $g$ with minimal value we may
assume that $g$ has integral coefficients in $K_{\infty}$ and at least
one of them has value $0\in \Gamma$. We write all its coefficients which
have value in $\Gamma$ in the form as given in our assumption. We take
$\tilde{s}$ to be the product of all appearing denominators. Then
$v\tilde{s}\in \Gamma$. After multiplication with $\tilde{s}$, all
coefficients of $g$ with value in $\Gamma$ are elements of
$S_{\infty}\,$, and there is at least one such coefficient. Now we write
$g(X_1,\ldots,X_i)=f(X_1,\ldots,X_i)+ h(X_1,\ldots,X_i)$ where every
coefficient of $f$ is in $S_{\infty}$ and has value less than $va_k$ for
some $k$, and every coefficient of $h$ has value bigger than $va_k$ for
all $k$ (we allow $h$ to be the zero polynomial). Since $g$ has
coefficients of value $v\tilde{s}$, the polynomial $f$ is nonzero.
Since $vy_i\geq 0$ for all $i$, we have that $vh(y_1,\ldots,y_i)$ is
bigger than $va_k$ for all $k$.

We choose $k$ such that the assumptions of Lemma~\ref{btl} hold; note
that $k$ exists since by our definition of $f$, the coefficient $c_f$
has value less than $va_k$ for some $k$. We obtain that
\[
vf(y_1,\ldots,y_i)\> < \> va_{E_i(k+1)}\> < \>vh(y_1,\ldots,y_i)\;.
\]
This gives that
\[
vg(y_1,\ldots,y_i)\>=\>v(f(y_1,\ldots,y_i)+h(y_1,\ldots,y_i))\>=\>
vf(y_1,\ldots,y_i)\> < \> va_{E_i(k+1)}\> < \>\infty\;,
\]
that is, $g(y_1,\ldots,y_i)\ne 0$.
\end{proof}

\parm
Now we are able to give the
\sn
{\bf Proof of Theorem~\ref{MTai}:}
\sn
In all cases of the proof, we will choose functions $\varphi$ that have
the previously required properties. We will choose a suitable sequence
$(b_k)_{k\in\N}$ of elements in $L$ and a sequence $(c_k)_{k\in\N}$ in
$K$. Then we will set $a_k:=c_kb_k$ and choose some values $\alpha_k\geq
va_k\,$.

\parm
First, let us consider the \underline{valuation-transcendental case}. We
set
\[
\varphi(k,\ell)\>:=\>k+k\ell\;,
\]
and note that equations (\ref{E}) now read as follows:
\begin{eqnarray*}
& E_1(k+1)\>\geq\>k+kE_1(k)+1,\\
& E_i(k+1)\>\geq\>E_{i-1}(k+kE_i(k)+1)\;.
\end{eqnarray*}
Further, we will work with a suitable element $t\in {\cal O}_L$
transcendental over $K$ and set, after a suitable choice of the sequence
$(c_k)_{k\in\N}\,$,
\begin{eqnarray*}
a_k &:= & c_k t^k\;,\\
\alpha_k &:= & va_k\;,\\
S_k &:= & K+Kt+\ldots+Kt^k\;.
\end{eqnarray*}
Conditions (A2) and (A3) are immediate consequences of our choice of
$S_k$ as the set of all polynomials in $K[t]$ of degree at most $k$.

\pars
Suppose that $vL/vK$ is not a torsion group. Then we pick $t\in
{\cal O}_L$ such that $vt$ is rationally independent over $vK$ (that is,
$nvt\notin vK$ for all integers $n>0$). Further, for all $k$ we set
$b_k=t^k$ and $c_k=1$ so that $a_k=t^k$. Then condition (A1) is
satisfied since we have that
\[
0\leq va_k=\alpha_k=vt^k= kvt<(k+1)vt=vt^{k+1}=va_{k+1}
\]
and
\[
k\alpha_{E_i(k)}=kva_{E_i(k)}=
kvt^{E_i(k)}=kE_i(k)vt<(k+kE_i(k))vt=\alpha_{\varphi(k,E_i(k))}\;.
\]

Suppose now that $vL/vK$ is a torsion group. In this case, $Kv|Lv$ is
transcendental by assumption, and we note that since $v$ is assumed
nontrivial on $L$, it must be nontrivial on $K$. We pick $t\in
{\cal O}_L$ such that $vt=0$ and $tv$ is transcendental over $Kv$.
Further, we choose a sequence $(c_k)_{k\in\N}$ in ${\cal O}_K$ such
that
\[
vc_{k+1} \> \geq\>kvc_k
\]
for all $k$. Since $va_k=vc_k+kvt=vc_k\,$, we obtain that
$va_k=\alpha_k<va_{k+1}$ and
\begin{equation}                            \label{kvak<1}
k\alpha_k\>=\>kva_k \> \leq \>va_{k+1}\;.
\end{equation}
Then by (\ref{k+1k}),
\begin{equation}                            \label{kvak<2}
k\alpha_{E_i(k)}\><\>E_i(k)\alpha_{E_i(k)}\>\leq\>va_{E_i(k)+1}\>\leq\>
va_{\varphi(k,E_i(k))}\>\leq\>\alpha_{\varphi(k,E_i(k))}\;.
\end{equation}
Hence again, condition (A1) is satisfied.

Now we have to verify (A4), simultaneously for all of the above choices
for $a_k\,$. Take $d_0,\ldots,d_m\in S_k\,$, $m\leq k$, and write
$d_j=\sum_{\nu=0}^{k} d_{j\nu}t^{\nu}$ with $d_{j\nu}\in K$. Then
\[
d_0+d_1a_{k+1}+\ldots+d_ma_{k+1}^m\>=\>\sum_{j=0}^{m}
\sum_{\nu=0}^{k} c_{k+1}^j d_{j\nu}t^{j(k+1)+\nu}\;.
\]
In this sum, each power of $t$ appears only once. So we have, by
Lemma~\ref{prelBour},
\[
v(d_0+d_1a_{k+1}+\ldots+d_ma_{k+1}^m)\>=\>\min_{j,\nu}
vc_{k+1}^j d_{j\nu}t^{j(k+1)+\nu}\>:=\>\beta\;.
\]
If this minimum is obtained at $j=j_0$ and $\nu=\nu_0\,$, then
\begin{eqnarray*}
\beta & = & vc_{k+1}^{j_0} d_{j_0\nu_0}t^{j_0(k+1)+\nu_0}\>=\>
\min_{\nu} vc_{k+1}^{j_0} d_{j_0\nu}t^{j_0(k+1)+\nu} \\
& = & (\min_{\nu} vd_{j_0\nu}t^{\nu})+va_{k+1}^{j_0}\>=\>
vd_{j_0}a_{k+1}^{j_0}\;,
\end{eqnarray*}
where the last equality again holds by Lemma~\ref{prelBour}. For all
$j$,
\[
\beta\>\leq\>
\min_{\nu} vc_{k+1}^{j} d_{j\nu}t^{j(k+1)+\nu}\>=\>(\min_{\nu}
vd_{j\nu}t^{\nu})+va_{k+1}^{j}\>=\> vd_{j}a_{k+1}^{j}\;.
\]
This gives that
\[
v(d_0+d_1a_{k+1}+\ldots+d_ma_{k+1}^m)\>=\>\beta\>=\>\min_j
vd_{j}a_{k+1}^{j}\>\leq\>vd_m+mva_{k+1}\>=\>vd_m+m\alpha_{k+1}\;,
\]
as required. Finally, we have to verify the assumption of
Corollary~\ref{btcor}. Each element in $K_{\infty}$ can be written as a
quotient $r/s$ of polynomials in $t$ with coefficients in $K$. After
multiplying both $r$ and $s$ with a suitable element from $K$ we may
assume that $s$ has coefficients in ${\cal O}_K$ and one of them is $1$.
If this is the coefficient of $t^i$, say, then it follows by
Lemma~\ref{prelBour} that $0\leq vs\leq vt^i\leq va_i$ and thus, $vs\in
\Gamma$.

Now we take any maximal immediate extension $(M,v)$ and $y_i$ as
defined preceeding to Lemma~\ref{btl}. Then we can infer from
Corollary~\ref{btcor} that the elements $y_i$ are algebraically
independent over $K_{\infty}\,$; that is, the transcendence degree of
$M$ over $K_{\infty}$ is infinite. Since the transcendence degree of
$L$ over $K$ and thus also that of $L$ over $K_{\infty}$ is finite, we
can conclude that the transcendence degree of $M$ over $L$ is infinite.

\parb
Next, we consider the \underline{value-algebraic case} and the
\underline{residue-algebraic case}. We will assume for now that there is
an algebraic subextension $L_0|K$ of $L|K$ such that $vL_0/vK$ contains
elements of arbitrarily high order, or $L_0v$ contains elements of
arbitrarily high degree over $Kv$. The remaining cases will be treated
at the end of the proof of our theorem.

For the present case as well as the separable-algebraic case, we work
with any function $\varphi$ that satisfies the conditions outlined in
the beginning of this section, and with
\[
S_k\>:=\>K(a_1,\dots,a_k)\;.
\]
Then $S_{\infty}$ is a field and the assumption of Corollary~\ref{btcor}
are trivially satisfied (taking $s=1$). Further, condition (A2) is
trivially satisfied. To prove that condition (A3) holds, take any $u\in
S_{\ell}= K(a_1,\dots,a_{\ell})\,$. If $n=\max\{k,\ell\}$, then
$d_0,\ldots,d_k,u\in K(a_1,\dots,a_n)=S_n$ and therefore,
\[
d_0+d_1u +\ldots+d_ku^k\>\in\> S_n\>\subseteq\> S_{\varphi(k,\ell)}\;.
\]
This shows that (A3) holds.

By induction, we define $a_k\in L_0$ as follows, and we always take
$\alpha_k=va_k$. We start with $a_1=1$ and $\alpha_1=0$. Suppose that
$a_1,\dots,a_k$ are already defined. Since $K(a_1,\dots,a_k)|K$ is a
finite extension, also $vK(a_1,\dots,a_k)/vK$ and $K(a_1,\dots,a_k)v|Kv$
are finite. Hence by our assumption in the algebraic case, there is some
$b_{k+1}\in L_0$ such that
\begin{eqnarray}
&& \mbox{$0,vb_{k+1},2vb_{k+1},\ldots,kvb_{k+1}$ lie in distinct cosets
modulo $vK(a_1,\dots,a_k)$, or} \label{va}\\
&& \mbox{$1,b_{k+1}v,(b_{k+1}v)^2,\ldots,
(b_{k+1}v)^k$ are $K(a_1,\dots,a_k)v$-linearly independent.} \label{ra}
\end{eqnarray}
If $L_0v$ contains elements of arbitrarily high degree over $Kv$, we
always choose $b_{k+1}$ such that (\ref{ra}) holds; in this case,
$vb_{k+1} =0$ and we choose the elements $c_k$ as in the
residue-transcendental case above. Otherwise, $vL_0/vK$ contains
elements of arbitrarily high order, and we always choose $b_{k+1}$ such
that (\ref{va}) holds. In this case, we choose $c_{k+1}$ such that for
$a_{k+1}:=c_{k+1}b_{k+1}$ we obtain $k\alpha_k=kva_k\leq va_{k+1}\,$;
this is possible since the values of $b_k$ and hence of all $a_k$ lie in
the convex hull of $vK$ in $vL$. As in the residue-transcendental case
above, we obtain (\ref{kvak<1}) and (\ref{kvak<2}), showing that
condition (A1) is satisfied.

To prove that (A4) holds, take any $k\geq 1$ and $d_0,\ldots,d_k\in S_k
=K(a_1,\dots,a_k)$. By Lemma~\ref{algvind} applied to $b_{k+1}\,$,
\begin{eqnarray*}
v(d_0+d_1a_{k+1}+\ldots+d_ka_{k+1}^k) & = &
v(d_0+d_1c_{k+1}b_{k+1}+\ldots+d_kc_{k+1}^kb_{k+1}^k)\\
 & = & \min_i vd_ic_{k+1}^ib_{k+1}^i\>=\>\min_i vd_ia_{k+1}^i\;.
\end{eqnarray*}
This shows that (A4) holds.

As in the valuation-transcendental case, we can now deduce
our assertion about the transcendence degree.

\parb
Next, we consider the \underline{separable-algebraic case}. In this
case, we can w.l.o.g.\ assume that $(K,v)$ is henselian. Indeed, each
maximal immediate extension of $(L,v)$ contains a henselization $L^h$ of
$(L,v)$ and hence also a henselization $K^h$ of $(K,v)$, and our
assumption on $L_0$ implies that the subfield $L_0.K^h$ of $L^h$ is an
infinite separable-algebraic extension of $K^h$. (Here, $L_0.K^h$
denotes the field compositum, i.e., the smallest subfield of $L^h$ which
contains $L_0$ and $K^h$.)

We take $S_k$ and $\varphi(k,\ell)$ as in the previous case, so that
again, (A2), (A3) and the assumption of Corollary~\ref{btcor} hold. Then
we take $a_1=b_1$ to be any element in ${\cal O}_{L_0}\setminus K$ and
choose some $\alpha_1\in vK$ such that $\alpha_1\geq\mbox{\rm kras}
(a_1,K) \in v\tilde{K}$; this is possible since $vK$ is cofinal in its
divisible hull, which is equal to $v\tilde{K}$. Inequality (\ref{krasa})
of Lemma~\ref{vf(a)} shows that $\mbox{\rm kras}(a_1,K)\geq
va_1\,$, so that $\alpha_1\geq va_1\,$. Suppose we have chosen $a_1,
\ldots,a_k\in {\cal O}_{L_0}\,$. Since $L_0|K$ is infinite and
separable-algebraic, the same is true for $L_0|K(a_1,\ldots,a_k)$. By
the Theorem of the Primitive Element, we can therefore find an element
$b_{k+1}\in L_0$ such that
\[
[K(a_1,\ldots,a_k,b_{k+1}):K(a_1,\ldots,a_k)]\>\geq \>k+1\;.
\]
We choose $c_{k+1}\in K$ such that for $a_{k+1}:=c_{k+1}b_{k+1}$ we have
that $k\alpha_k\leq va_{k+1}\,$. Finally, we choose $\alpha_{k+1}\in vK$
such that
\[
\alpha_{k+1}\>\geq\>\mbox{\rm kras}(a_{k+1},K)\>\geq va_{k+1}\>.
\]
Again, we obtain that (\ref{kvak<2}) and (A1) hold.

It only remains to show that (A4) holds. But this follows readily
from inequality (\ref{krasfa}) of Lemma~\ref{vf(a)}, where we take
$K(a_1,\ldots,a_k)$ in place of $K$ and $a=a_{k+1}$, together with the
fact that $\mbox{\rm kras}(a_{k+1},K(a_1\ldots,a_k))\leq
\mbox{\rm kras}(a_{k+1},K)$.

As before, we now obtain our assertion about the transcendence degree.

\parb
It remains to prove the \underline{value-algebraic case} and the
\underline{residue-algebraic case} for transcendental valued field
extensions $(L|K,v)$ of finite transcendence degree. We assume that
$vL/vK$ is a torsion group containing elements of arbitrarily high order
or the extension $Lv|Kv$ is algebraic and such that $Lv$ contains
elements of arbitrarily high degree over $Kv$.

Take any subextension $E|K$ of $L|K$. Then $(L|K,v)$ satisfies the
above assumption if and only if at least one of the extensions $(L|E,v)$
and $(E|K,v)$ satisfies the assumption. Choose a transcendence basis
$(x_1,\ldots,x_n)$ of $L|K$ and set
\[
F\>:=\>K(x_1,\ldots,x_n)\>.
\]
Then $L|F$ is algebraic. By what we have already proved, if $vL/vF$
contains elements of arbitrarily high order or $Lv$ contains elements of
arbitrarily high degree over $Fv$, then any maximal immediate extension
of $(L,v)$ has infinite transcendence degree over $L$.

Suppose now that $(F|K,v)$ satisfies the assumption on the value group or
the residue field extension. Take $s\in\N$ minimal such that
$vK(x_1,\ldots,x_s)/vK$ contains elements of arbitrarily high order or
$K(x_1,\ldots,x_s)v$ contains elements of arbitrarily high degree over
$Kv$. Then the assertion holds also for the value group or the residue
field extension of $(K(x_1,\ldots,x_s)|K(x_1,\ldots,x_{s-1}),v)$. We can
replace $K$ by $K(x_1,\ldots,x_{s-1})$ and we will write $x$ in place of
$x_s\,$ so that now we have a subextension $(K(x)|K,v)$ that satisfies
the assertion for its value group or its residue field extension.

In both the value-algebraic and the residue-algebraic case we define
$b_k\in K[x]$ by induction on $k$ and set
\[
S_k\>:=\>K[x]_{N_k}\quad \mbox{with $N_k:= \deg b_k$.}
\]

Assume that $vK(x)$ contains elements of arbitrarily high order modulo
$vK$. Then such elements can be already chosen from $vK[x]$. We set
$b_1=1$. Suppose that $b_1,\ldots,b_k$ are already chosen with $\deg
b_{i-1}< \deg b_i$ for $1<i\leq k$. From Corollary~\ref{polvspaces} we
know that $vS_k$ contains only finitely many values that represent
distinct cosets modulo $vK$. Since all of these values are torsion
modulo $vK$, the subgroup $\langle vS_k \rangle$ of $vK(x)$ generated by
$vS_k$ satisfies $(\langle vS_k \rangle:vK)<\infty$.
%
By assumption, there is $b_{k+1}\in K[x]$ for which the order of
$vb_{k+1}$ modulo $vK$ is at least $(k+1)(\langle vS_k \rangle:vK)$;
this forces $0,vb_{k+1},2vb_{k+1},\ldots,kvb_{k+1}$ to lie in distinct
cosets modulo $\langle vS_k \rangle$. Since $b_{k+1}\notin
K[x]_{N_k}\,$, we have that $N_{k+1}=\deg b_{k+1}>N_k\,$.

\pars
Assume now that $K(x)v$ contains elements of arbitrarily high degree
over $Kv$. Without loss of generality we can assume that $vK(x)/vK$
is then a torsion group with a finite exponent $N$. Otherwise, $vL/vK$ is
not a torsion group and we are in the valuation-transcendental case or
$vK(x)/vK$ contains elements of arbitrarily high order and we are in
the value-algebraic case.

The elements of arbitrarily high degree over $Kv$ can be chosen from
$K[x]v$. Indeed, suppose there is $m\in\N$ such that $[Kv(fv):Kv]\leq m$
for every polynomial $f$ of nonnegative value. Take any $r=\frac{h}{g}$,
where $g,h\in K[x]$ and $vr=0$. By the assumption on
$vK(x)/vK$ we have that $nvh=vd$ for some natural number $n\leq N$
and $d\in K$. Then
\[
r\>=\> \frac{d^{-1}h^n}{d^{-1}h^{n-1}g}
\]
and $vd^{-1}h^{n-1}g=vd^{-1}h^n=0$, since $vh=vg$. Therefore we may
assume that $vh=vg=0$. Hence,
\[
[ Kv (rv):Kv ]\>\leq\> [ Kv(rv, gv):Kv ] \>=\>
[ Kv(hv,gv):Kv ]\leq m^2
\]
for every $r\in K(x)$ with $vr=0$, a contradiction to our assumption.

As in the value-algebraic case, we set $b_1=1$. Suppose that
$b_1,\ldots,b_k$ are already chosen with $\deg b_{i-1}< \deg b_i$
for $1<i\leq k$. By Corollary~\ref{polvspaces},
there are at most $NN_k+1$ many $Kv$-linearly independent elements in
$K[x]_{NN_k}v$, and as all of them are algebraic over $Kv$, it follows
that the extension $Kv\left(K[x]_{NN_k}v\right)|Kv$ is finite. By
assumption, there is $b_{k+1}\in K[x]$ such that $vb_{k+1}=0$ and the
degree of $b_{k+1}v$ over $Kv$ is at least $(k+1)[Kv\left(K[x]_{NN_k}v
\right):Kv]$, which forces the elements $1,b_{k+1}v,(b_{k+1}v)^2,\ldots,
(b_{k+1}v)^k$ to be $Kv\left(K[x]_{NN_k}v\right)$-linearly independent.
Since $b_{k+1}\notin K[x]_{NN_k}\,$, we have that $N_{k+1}=\deg b_{k+1}>
NN_k\geq N_k\,$.

\pars
For the value-algebraic as well as for the residue-algebraic case we set
\[
\varphi (k,l):=N_k+N_kN_l.
\]
Since in both cases $(N_k)_{k\in\N}$ is a strictly increasing sequence
of natural numbers, $\varphi$ has the required properties. As in the
first part of the proof of the value-algebraic and the residue-algebraic
case, one can show that the elements $c_k\in K$ can be chosen in such a
way that condition (A1) holds for $a_k:=c_kb_k$ and $\alpha_k:=va_k$.
Since $(N_k)_{k\in\N}$ is strictly increasing, condition (A2) is
trivially satisfied. Moreover, $N_k\geq k$ for every $k\in\N$. Hence for
any $d_0,\ldots,d_k\in S_k$ and $u\in S_l\,$,
\[
\deg (d_0+d_1u+\cdots + d_ku^k)\leq N_k+kN_l\>\leq\> \varphi (k,l)\>\leq\>
N_{\varphi (k,l)}\>.
\]
Thus, $d_0+d_1u+\cdots + d_ku^k\in S_{\varphi (k,l)}$. This shows that
(A3) holds.

\pars
To verify (A4), we take any $k,m\in\N$ with $m\leq k$, and $d_0,\ldots,
d_m \in S_k$. We wish to estimate the value of the element $d_0+
d_1a_{k+1} +\cdots d_ma_{k+1}^m$. We discuss first the value-algebraic
case. Note that the values $v(d_ia_{k+1}^i)$, $0\leq i\leq m$, lie in
distinct cosets modulo $vK$. Indeed, $vd_ia_{k+1}^i=
vd_ic_{k+1}^i+ivb_{k+1}$, where $d_ic^i_{k+1}\in S_k$. Therefore, if
\[
vd_ia_{k+1}^i+vK\>=\>vd_ja_{k+1}^j+vK
\]
for some $0\leq i\leq j\leq m$, then also
\[
ivb_{k+1}+\langle vS_k \rangle\>=\>jvb_{k+1}+\langle vS_k \rangle\>,
\]
which by our choice of $b_{k+1}$ yields that $i=j$. Hence, from Lemma
\ref{algvind} it follows that
\[
v(d_0+d_1a_{k+1}+\cdots d_ma_{k+1}^m)\>=\> \min_{i} vd_ia_{k+1}^i\>\leq\>
vd_m+mva_{k+1}\>.
\]

We obtain the same assertion also in the residue-algebraic case. If
$d_i =0$ for all $i$, then it is trivially satisfied. If not, take $i_0$
so that
\[
vd_{i_0}c_{k+1}^{i_0}\>=\> \min_i vd_ic_{k+1}^i\>=\>\min_i
vd_ia_{k+1}^i\>.
\]
We have that $vd_{i_0}^N=vc$ for some $c\in K$. Setting $d:=c^{-1}
c_{k+1}^{-i_0}d_{i_0}^{N-1}\,$, we obtain that
\[
v(d_0+d_1a_{k+1}+\cdots d_ma_{k+1}^m)\>= -vd +v\xi
\]
with $\xi:= dd_0+dd_1c_{k+1}b_{k+1}+\cdots+dd_m c^m_{k+1}b_{k+1}^m$.
Note that $dd_i\in K[x]_{NN_k}$ for $0\leq i\leq m$, and that
\[
vdd_i c^i_{k+1}\>\geq \>vdd_{i_0} c^{i_0}_{k+1}\>=\>vc^{-1}d_{i_0}^N
\>=\>0.
\]
In particular, $v\xi\geq 0$, and
\[
\xi v=(dd_0)v+(dd_1c_{k+1})vb_{k+1}v+\cdots+(dd_m c^m_{k+1})v(b_{k+1}v)^m
\]
is a linear combination of $1,b_{k+1}v,(b_{k+1}v)^2,\ldots,(b_{k+1}v)^m$
with coefficients from $Kv\left(K[x]_{N\cdot N_k}v\right)$. Since at
least one of them, the element $dd_{i_0} c^{i_0}_{k+1}v$, is nonzero,
also the linear combination is nontrivial by our choice of $b_{k+1}$.
Hence $v\xi=0$ and
\[
v(d_0+d_1a_{k+1}+\cdots d_ma_{k+1}^m)\>=\> -vd \>=\>vd_{i_0}c_{k+1}^{i_0}
\> \leq \> vd_m+mva_{k+1}\>.
\]
Therefore, condition (A4) is satisfied in both cases.

\pars
It suffices now to verify the assumptions of Corollary \ref{btcor}. Take
any element $\frac{h}{g}$ of $K_{\infty}=K(x)$, where $g,h\in
S_{\infty}= K[x]$. In both the value-algebraic and the residue-algebraic
case we assumed that $vK(x)/vK$ is a torsion group. Therefore, as in the
residue-algebraic case above one can multiply $h$ and $g$ by a
suitable polynomial to obtain that $vg=0\in \Gamma $. Hence the
assumptions of the corollary are satisfied.

Since the transcendence degree of the extension $L|K(x)$ is finite,
we can now deduce the assertion about about the transcendence degree as
in the previous cases.

\parm
In the \underline{value-algebraic case}, we still have to deal with the
case where there is a subgroup $\Gamma\subseteq vL$ containing $vK$ such
that $\Gamma/vK$ is an infinite torsion group and the order of each of
its elements is prime to the characteristic exponent of $Kv$.
We may assume that $Lv|Kv$ is algebraic since otherwise, the assertion
of our theorem follows from the valuation-transcendental case. Since
every maximal immediate extension of $(L,v)$ contains a henselization
of $(L,v)$, we may assume that both $(L,v)$ and $(K,v)$ are henselian.
We take $L'$ to be the relative separable-algebraic closure of $K$ in
$L$. Then by Lemma~\ref{vrrac}, $vL/vL'$ is a $p$-group, which yields
that $\Gamma\subseteq vL'$. In view of the fundamental inequality, we
find that $L'|K$ must be an infinite extension. Now the assertion of our
theorem follows from the separable-algebraic case.

\parb
Finally, we have to deal with our additional assertion about the
completion. Since the transcendence degree of $L|K$ is finite, we know
that $vL/vK$ has finite rational rank. Therefore, $vK$ is cofinal in
$vL$ or there exists some $\alpha\in vL$ such that the sequence
$(i\alpha)_{i\in\N}$ is cofinal in $vL$. In the latter case (which
always holds if $vL$ contains an element $\gamma$ such that
$\gamma>vK$), we are in the value-transcendental case and we choose the
element $t$ such that $vt=\alpha$. In the former case, provided that the
cofinality of $vL$ is countable, we choose the elements $c_i$ such that
the sequence $(vc_ib_i)_{i\in\N}$ is cofinal in $vL$. In all of these
cases, the sequence $(va_i)_{i\in\N}$ will be cofinal in $vL$ and the
elements $y_i$ will lie in the completion of $(L,v)$.
\QED

%
%
\section{Extensions of maximal fields}      \label{sectemf}
We start with the
\sn
{\bf Proof of Theorem~\ref{MT2}}:
\sn
Take a maximal field $(K,v)$ which satisfies (\ref{finpdeg}), and denote
by $p$ the characteristic exponent of $Kv$. Further, take an infinite
algebraic extension $(L|K,v)$.
Denote the relative separable-algebraic closure of $K$ in $L$ by $L'$.
Assume that $L'|K$ is infinite. Since $K$ is henselian, the
separable-algebraic case of Theorem \ref{MTai} shows that any maximal
immediate extension of $(L,v)$ has infinite transcendence degree over
$L$.

Assume now that $L'|K$ is a finite extension. Then the field $(L',v)$
is maximal, $(vL':pvL')[L'v:L'v^p]=(vK:pvK)[Kv:Kv^p]<\infty$, and $L|L'$
is an infinite purely inseparable extension. Therefore at least one of
the extensions $vL|vL'$ or $Lv|L'v$ is infinite. Indeed, suppose that
$(vL:vL')$ and $[Lv:L'v]$ were finite. Take any finite subextension
$E|L'$ of $L|L'$ such that $[E:L']>(vL:vL')[Lv:L'v]$. Then
\[
[E:L']\> >\>(vL:vL')[Lv:L'v]\>\geq\> (vE:vL')[Ev:L'v]\>,
\]
which contradicts the fact that $L'$ as a maximal field is defectless
by Theorem~\ref{mfhdl}. If $vL/vL'$ contains elements of arbitrarily
high order or $Lv$ contains elements of arbitrarily high degree over
$L'v$, then from the value-algebraic or residue-algebraic case of
Theorem \ref{MTai} we deduce that any maximal immediate extension of $L$
is of infinite transcendence degree over $L$. Otherwise, $vL/vL'$ is a
$p$-group of finite exponent, $Lv|L'v$ is a purely inseparable extension
with $(Lv)^{p^n}\subseteq L'v$ for some natural number $n$, and $vL/vL'$
or $Lv|L'v$ is infinite. But this is not possible if $[L'v:L'v^p]
(vL':pvL')<\infty$.                                   \QED

\pars
For the proof of Theorem~\ref{infpdeg} we will need the following
result:

\begin{lemma}                               \label{1/pmax}
If $(K,v)$ is a maximal field of characteristic $p>0$, then also
$K^{1/p}$ with the unique extension of the valuation $v$ is a maximal
field.
\end{lemma}
\begin{proof}
If $(a_{\nu})$ is a pseudo Cauchy sequence in $L$, then $(a_{\nu}^p)$ is
a pseudo Cauchy sequence in $K$. Since $(K,v)$ is maximal, it has a
pseudo limit $b\in K$. But then, $a=b^{1/p}\in L$ is a pseudo limit of
$(a_{\nu})$.
\end{proof}

\pars
After this preparation, we can give the
\sn
{\bf Proof of Theorem~\ref{infpdeg}}:
\sn
Part a) follows immediately from Lemma~\ref{1/pmax}.

\pars
To prove assertions b) and c) we consider the following subsets of $K$.
We take $A$ to be a set of elements of $K$ such that the cosets
$\frac{1}{p}va+vK$, $a\in A$, form a basis of the $\Z/p\Z$-vector space
$\frac{1}{p}vK/vK$. Similarly, we take $B$ to be a set of elements of
the valuation ring of $(K,v)$ such that the residues $(bv)^{1/p}$, $b\in
B$, form a basis of $(Kv)^{1/p}|Kv$.
Then
\[
\frac{1}{p}vK\>=\>vK+\sum_{a\in A} \frac{1}{p}va\Z \;\;\mbox{ and }\;\;
(Kv)^{1/p}\>=\>Kv((bv)^{1/p}\mid b\in B)\>.
\]

In order to prove assertion b) of our theorem, we set
\[
L_{\infty}\>:=\> K(a^{1/p}, b^{1/p}\mid a\in A,\,b\in B)\>\subseteq\>
K^{1/p}
\]
and obtain that $vL_{\infty}=\frac{1}{p}vK$ and $L_{\infty}v=
(Kv)^{1/p}$. So the extension $(K^{1/p}|L_{\infty},v)$ is immediate.
Lemma~\ref{1/pmax} shows that $(K^{1/p},v)$ is a maximal immediate
extension of $(L_{\infty},v)$. Our goal is now to show that under the
assumptions of the theorem, this extension is of degree
at least $\kappa$. Once this is proved, we can take $X\subseteq K^{1/p}$
to be a minimal set of generators of the extension $K^{1/p}|L_{\infty}$.
Then the elements of $X$ are $p$-independent over $L_{\infty}$. Take any
natural number $n$. As $X$ is infinite, we can choose $x_1,\ldots,x_n\in
X$ and set $L_n:=L_{\infty}(X\setminus \{x_1, \ldots,x_n\})$. Then
$K^{1/p}|L_n$ is an immediate extension of degree $p^n$. Similarly,
for $\lambda$ any infinite cardinal $\leq\kappa$, take $Y\subseteq X$ of
cardinality $\lambda$ and set $L_\lambda:=L_{\infty}(X\setminus Y)$.
Then $K^{1/p}|L_\lambda$ is an immediate algebraic extension of degree
$\lambda$.

\pars
We assume first that $\kappa=(vK:pvK)$, so the set $A$ is infinite.
Then we take a partition of $A$ into $\kappa$ many countably infinite
sets $A_{\tau}$, $\tau<\kappa$. We choose enumerations
\[
A_{\tau}\>=\>\{a_{\tau,i} \mid i\in\N\} \>.
\]

For every $\mu<\kappa$ we set ${\cal A}_\mu:=\bigcup_{\tau<\mu} A_\tau$
and
\[
K_\mu\>:=\> K(a^{1/p}\mid a\in {\cal A}_\mu)\>.
\]
Note that ${\cal A}_0=\emptyset$ and $K_0=K$. We claim that
\begin{equation}                            \label{vLmu}
vK_\mu\>=\>vK + \sum_{a\in {\cal A}_\mu}\frac{1}{p}va\Z
\;\;\mbox{ and }\;\; K_\mu v\>=\> Kv\>.
\end{equation}
The inclusions ``$\supseteq$'' are clear. For the converses, we observe
that value group and residue field of $K_\mu$ are the unions of the
value groups and residue fields of all finite subextensions of $K_\mu
|K$. Such subextensions can be written in the form $F=K(a_1,\ldots,a_k)$
with distinct $a_1,\ldots,a_k\in {\cal A}_\mu\,$, and we have that
\[
p^k\>\geq\>[F:K]\>\geq\> (vF:vK)[Fv:Kv]\>\geq\> p^k\cdot 1\>,
\]
so equality holds everywhere. Consequently, $vF=vK + \sum_{i=1}^k
va_i\Z$ and $Fv=Kv$. This proves our claim.

\pars
For every $\tau<\kappa$ we choose a sequence $(c_{\tau,i})_{i\in\N}$
of elements in $K$ such that the sequence of values
\begin{equation}                            \label{sv}
(vc_{\tau,i}^{ } a_{\tau,i}^{1/p})_{i\in\N}
\end{equation}
is strictly increasing. If the cofinality of $vK$ is countable, then the
elements $c_{\tau,i}$ can be chosen in such a way that the sequence
(\ref{sv}) is cofinal in $\frac{1}{p}vK$. For every $n\in\N$, we set
\begin{equation}                                     
\xi_{\tau,n}\>:=\>\sum_{i=1}^n c_{\tau,i}^{ } a_{\tau,i}^{1/p}\>\in\>
K_{\tau+1}\>.
\end{equation}
Then $(\xi_{\tau,n})_{n\in\N}$ is a pseudo Cauchy sequence, hence it
admits a pseudo limit $\xi_{\tau}$ in the maximal field $(K^{1/p},v)$.
In order to show that the degree of $K^{1/p}|L_{\infty}$ is at least
$\kappa$, we prove by induction that for every $\mu<\kappa$ and each
$K'$ such that $K_{\mu+1}\subseteq K'\subseteq L_{\infty}$,
the pseudo Cauchy sequence $(\xi_{\mu,n})_{n\in\N}$ admits no pseudo limit
in $K'(\xi_{\tau}\mid\tau<\mu)$
and the extension
\begin{equation}                            \label{exti}
(K'(\xi_\tau\mid\tau\leq\mu)|K',v)
\end{equation}
is immediate.

Take $\mu<\kappa$ and assume that our assertions have already been
shown for all $\mu'<\mu$. If $\mu=\mu'+1$ is a successor ordinal, then
from (\ref{exti}) we readily get that the extension
\begin{equation}                            \label{exti1}
(K'(\xi_{\tau}\mid\tau<\mu)|K',v)
\end{equation}
is immediate for every $K'$ such that $K_\mu\subseteq K'\subseteq
L_\infty\,$. If $\mu$ is a limit ordinal, then (\ref{exti1}) follows
from the induction hypothesis since $K_{\mu'}\subseteq K_\mu\subseteq
K'$ for each $\mu'<\mu$ and since the union over the increasing chain
of immediate extensions $K'(\xi_{\tau}\mid\tau\leq\mu')$, $\mu'<\mu$,
of $(K',v)$ is again an immediate extension of $(K',v)$.

In order to prove the induction step,
%
%
suppose towards a contradiction that $(\xi_{\tau,n})_{n\in\N}$ admits a
pseudo limit $\eta_{\mu }$ in $K'(\xi_{\tau}\mid\tau<\mu)$ for some $K'$
such that $K_{\mu+1}\subseteq K'\subseteq L_{\infty}$.
%
%
Then $\eta_\mu$ lies already in a finite extension
\begin{equation}                                      \label{K'}
E\>:=\>K_\mu(\xi_\tau\mid\tau<\mu)(a_1^{1/p},\ldots,
a_k^{1/p}, b_1^{1/p},\ldots, b_\ell^{1/p})
\end{equation}
of $K_\mu(\xi_\tau\mid\tau<\mu)$ in $L_{\infty}(\xi_\tau\mid\tau<
\mu)$, with distinct elements $a_1,\ldots,a_k\in A\setminus
{\cal A}_\mu$ and $b_1,\ldots,b_\ell\in B$. We claim that
\begin{equation}                                      \label{vK'}
vE\>=\>vK_\mu+\sum_{i=1}^k\frac{1}{p}va_i\Z
\;\;\mbox{ and }\;\;
Ev\>=\>K_\mu v((b_1v)^{1/p},\ldots,(b_\ell v)^{1/p})\>.
\end{equation}
As the extension (\ref{exti1}) is immediate for $K_\mu$ in place of
$K'$, the inclusions ``$\supseteq$'' are clear. Conversely, from these
inclusions together with the equations in (\ref{vLmu}) and our
assumption on the $a_i\,$, it follows that $(vE:vK) \geq p^k$ as well as
$[Ev:Kv]\geq p^\ell$. Therefore, we have that $p^k\cdot p^\ell\geq
[E:K]\geq (vE:vK) [Ev:Kv] \geq p^k\cdot p^\ell$, so equality holds
everywhere. Consequently, $(vE:vK)= p^k$ and $[Ev:Kv] =p^\ell$, which
proves that the inclusions are equalities.

Now we take $n$ to be the minimum of all $i\in\N$ such that $a_{\mu,i}$
is not among the $a_1,\ldots,a_k\,$. We set $\xi_E:=0$ if $n=1$, and
$\xi_E:=\xi_{\mu,n-1}$ otherwise. Then $\eta_\mu -\xi_E\in E$. In
contrast, the fact that $\eta_\mu$ is a pseudo limit, together with
the first equation of (\ref{vK'}), yields that
\[
v(\eta_\mu -\xi_E)\>=\>v(\xi_{\mu,n}-\xi_E)\>=\>
vc_{\mu,n}+\frac{1}{p}va_{\mu,n}\>\notin\>
vK_\mu+ \sum_{i=1}^k \frac{1}{p}va_i\Z \>=\>vE\>.
\]
This contradiction proves that $(\xi_{\mu,n})_{n\in\N}$ admits no pseudo
limit in $K'(\xi_{\tau}\mid\tau<\mu)$. Thus in particular,
 $\xi_{\mu}\notin K'(\xi_{\tau}\mid\tau<\mu)$. Since $K_{\mu+1}\subseteq
K'$, $(\xi_{\mu,n})_{n\in\N}$ is a pseudo Cauchy sequence in
$(K'(\xi_\tau\mid\tau<\mu),v)$. As $[K'(\xi_\tau\mid \tau<\mu)
(\xi_\mu): K'(\xi_\tau\mid\tau<\mu)]$ $=p$ is a prime, Lemma~\ref{pimm}
shows that the extension $(K'(\xi_\tau\mid\tau\leq\mu)
|K'(\xi_\tau\mid\tau <\mu),v)$ is immediate. As also the extension
(\ref{exti1}) is immediate, we find that the extension
$(K'(\xi_\tau\mid\tau\leq\mu)|K',v)$ is immediate. This completes our
induction step. Because every extension $L_\infty(\xi_{\tau}\mid
\tau\leq\mu)|L_\infty(\xi_{\tau}\mid\tau<\mu)$ is nontrivial, it
follows that the degree of $K^{1/p}|L_{\infty}$ is at least $\kappa$.

\parm
A simple modification of the above arguments allows us to show the
assertion of part b) of the theorem in the case of $\kappa=[Kv:(Kv)^p]$,
in which the set $B$ is infinite. Let us describe these modifications.

We take a partition of $B$ into $\kappa$ many countably infinite sets
$B_{\tau}$, $\tau<\kappa$, and choose enumerations
\[
B_{\tau}\>=\>\{b_{\tau,i} \mid i\in\N\} \>.
\]
For every $\mu<\kappa$ we set ${\cal B}_\mu:=\bigcup_{\tau<\mu} B_\tau$
and
\[
K_\mu\>:=\> K(b^{1/p}\mid b\in {\cal B}_\mu)\>.
\]
Similarly as before, it is shown that
\begin{equation}                            \label{vL'mu}
vK_\mu\>=\>vK \;\;\mbox{ and }\;\; K_\mu v\>=\>
Kv((bv)^{1/p})\mid b\in {\cal B}_\mu)\>.
\end{equation}

We choose a sequence $(c_i)_{i\in\N}$ of elements in $K$ with strictly
increasing values. Again, if the cofinality of $vK$ is countable, then
the elements $c_i$ can be chosen in such a way that the sequence of
their values is cofinal in $vK$. For every $\tau<\kappa$ and $n\in\N$,
we set
\begin{equation}                                   \label{xi_resf}
\xi_{\tau,n}\>:=\>\sum_{i=1}^n c_i^{ }b_{\tau,i}^{1/p}\>\in\>
K_{\tau+1}\>.
\end{equation}
Now the only further part of the proof that needs to be modified is the
one that shows that $\eta_\mu\in E$, where $\eta_{\mu}$ is a pseudo
limit of $(\xi_{\mu,n})_{n\in\N}$, leads to a contradiction. In the
present case, we take $n$ to be the minimum of all $i\in\N$ such that
$b_{\mu,i}$ is not among the $b_1,\ldots,b_\ell\,$. As before, we set
$\xi_E:=0$ if $n=1$, and $\xi_E:=\xi_{\mu,n-1}$ otherwise. Then $\eta_\mu
-\xi_{E}\in E$. In contrast, the fact that $\eta_\mu$ is a pseudo
limit, together with the second equation of (\ref{vK'}), yields that
\begin{eqnarray*}
c_n^{-1}(\eta_\mu-\xi_E)v &=& c_n^{-1}(\xi_{\mu,n}-\xi_E)v
\>=\> (b_{\mu,n}^{1/p})v \>=\>(b_{\mu,n}v)^{1/p} \\
 & \notin & K_\mu v((b_1v)^{1/p},\ldots,(b_\ell v)^{1/p})\>=\>Ev\>,
\end{eqnarray*}
a contradiction. This completes our modification and thereby the proof
that the extension $(K^{1/p}|L_\infty,v)$ is of degree at least
$\kappa$.

\parb
We now turn to part c) of the theorem. Again, we consider separately the
cases of $\kappa=(vK:pvK)$ and of $\kappa=[Kv:(Kv)^p]$.

We assume first that $\kappa=(vK:pvK)$ and take a partition of $A$ as in
the proof of part b). Further, we set $s(1)=0$ and $s(m)=1+2+\cdots+
(m-1)$ for $m>1$. For every $\tau <\mu$ and every $m\in \N$, we set
\[
z_{\tau,m}\>:=\> \sum_{i=1}^m d_{\tau,s(m)+i}
a_{\tau,s(m)+i}^{p^{-i}}\in K^{1/p^{\infty}}\>,
\]
where $d_{\tau,j}$ are elements from $K$ such that for every $m\in\N$,
\sn
1) the sequence $(vd_{\tau,s(m)+i}^{ }a_{\tau,s(m)+i}^{p^{-i}}
)_{1\leq i\leq m}$ is strictly increasing,
\sn
2) $vd_{\tau,s(m)+m}^{ }a_{\tau,s(m)+m}^{p^{-m}}<
vd_{\tau,s(m+1)+1}^{ }a_{\tau,s(m+1)+1}^{p^{-1}}\,$.
\sn
If the cofinality of $vK$ is countable, then the elements $d_{\tau,i}$
can be chosen in such a way that the sequence that results from the
above is cofinal in $vK$.

We note that $z_{\tau,m}^{p^m}\in K$ with
\begin{equation}                            \label{ztm}
[K(z_{\tau,m}):K] \>=\> p^m \;\;\mbox{ and }\;\;
\frac{1}{p}va_{\tau,s(m)+1},\ldots,\frac{1}{p}
va_{\tau,s(m)+m} \in vK(z_{\tau,m})\>.
\end{equation}
We set
\[
L_\mu\>:=\>K(z_{\tau,m} \mid\tau<\mu,\, m\in\N)
\;\mbox{ for } \mu\leq\kappa\,,\mbox{ and } \;\; L\>:=\>L_\kappa\>.
\]
Further, we fix a maximal immediate extension $(M,v)$ of $(L,v)$. We
claim that
\begin{equation}                            \label{vgrsLmu}
vL_\mu\>=\>vK + \sum_{a\in {\cal A}_\mu} \frac{1}{p}va
\;\;\mbox{ and }\;\; L_\mu v\>=\>Kv\>.
\end{equation}
In particular, this shows that
\begin{equation}                            \label{vLLv}
vL=\frac{1}{p}vK \;\;\mbox{ and }\;\; Lv\>=\>Kv\>.
\end{equation}

To prove our claim, we observe that the first inclusion ``$\supseteq$''
in (\ref{vgrsLmu}) follows from (\ref{ztm}). We choose any $\mu<\kappa$,
$k\in\N$, $\tau_1,\ldots,\tau_k<\mu$ and $m_1,\ldots, m_k\in \N$ such
that the pairs $(\tau_i,m_i)$, $1\leq i\leq k$, are distinct. Then we
compute, using (\ref{ztm}):
\begin{eqnarray*}
p^{m_1}\cdot\ldots\cdot p^{m_k} & \geq & [K(z_{\tau_1,m_1},\ldots,
z_{\tau_k,m_k}):K] \\
& \geq & (vK(z_{\tau_1,m_1},\ldots,z_{\tau_k,m_k}):vK)
[K(z_{\tau_1,m_1},\ldots,z_{\tau_k,m_k})v:Kv] \\
& \geq & (v K(z_{\tau_1,m_1},\ldots,z_{\tau_k,m_k}):vK)\\
& \geq & (vK+ \sum_{j=1}^k\sum_{i=1}^{m_j}
\frac{1}{p}va_{\tau_j,s(m_j)+i}\Z\,:vK)\>\geq\>
p^{m_1}\cdot\ldots\cdot p^{m_k}\>,
\end{eqnarray*}
showing that equality holds everywhere. Therefore,
\[
vK(z_{\tau_1,m_1},\ldots,z_{\tau_k,m_k}) \>=\> vK+
\sum_{j=1}^k\sum_{i=1}^{m_j} \frac{1}{p}va_{\tau_j,s(m_j)+i}\Z
\>\subseteq\> vK + \sum_{a\in {\cal A}_\mu} \frac{1}{p}va
\]
and
\[
K(z_{\tau_1,m_1},\ldots,z_{\tau_k,m_k})v\>=\>Kv\>.
\]
Since the value group and residue field of $L_\mu$ are the unions of the
value groups and residue fields of all subfields of the above form, this
proves our claim.

\pars
For every $\tau<\kappa$ and $n\in\N$, we set
\[
\zeta_{\tau,n}:=\sum_{m=1}^n z_{\tau,m}\in L.
\]
Then $(\zeta_{\tau,n})_{n\in\N}$ is a pseudo Cauchy sequence in $(L,v)$,
hence it admits a pseudo limit $\zeta_{\tau}$ in the maximal field
$(M,v)$. In order to show that the transcendence degree of $M|L$ is at
least $\kappa$, we prove by induction that for every $\mu<\kappa$
and every field $L'$ such that $L_{\mu+1}\subseteq L'\subseteq L$, the
pseudo Cauchy sequence $(\zeta_{\mu,n})_{n\in\N}$ is of transcendental
type over $L'(\zeta_{\tau}\mid\tau<\mu)$, so that the extension
$(L'(\zeta_{\tau}\mid\tau\leq\mu)|L'(\zeta_{\tau}\mid\tau<\mu),v)$ is
immediate and transcendental and then also the extension
\begin{equation}                            \label{extia}
(L'(\zeta_{\tau}\mid\tau\leq\mu)|L',v)
\end{equation}
is immediate.

Take $\mu<\kappa$ and assume that our assertions have already been
shown for all $\mu'<\mu$. If $\mu=\mu'+1$ is a successor ordinal, then
from (\ref{extia}) we readily get that the extension
\begin{equation}                            \label{extib}
(L'(\zeta_{\tau}\mid\tau<\mu)|L',v)
\end{equation}
is immediate for every $L'$ such that $L_\mu\subseteq L'\subseteq L$.
If $\mu$ is a limit ordinal, then (\ref{extib}) follows from the
induction hypothesis since $L_{\mu'}\subseteq L_\mu\subseteq L'$ for
each $\mu'<\mu$ and since the union over an increasing chain of
immediate extensions of $(L',v)$ is again an immediate extension of
$(L',v)$.

In order to prove the induction step, take any $L'$ such that $L_{\mu+1}
\subseteq L'\subseteq L$. Suppose towards a contradiction that the
pseudo Cauchy sequence $(\zeta_{\mu,n})_{n\in\N}$ in $L_{\mu+1}$ is of
algebraic type over $(L'(\zeta_\tau\mid\tau<\mu),v)$ (which includes the
case where it has a pseudo limit in $L'(\zeta_\tau\mid\tau<\mu)$). Then
by Theorem~\ref{KT3} there exists an immediate algebraic extension
$(L'(\zeta_\tau\mid\tau<\mu)(d)|L'(\zeta_\tau\mid \tau<\mu),v)$ with
$d$ a pseudo limit of the sequence. The element $d$ is also algebraic
over $L_\mu(\zeta_\tau\mid\tau<\mu)$.
On the other hand, we will now show that from the fact that $d$ is a
pseudo limit of $(\zeta_{\mu,n})_{n\in\N}$ it follows that the value
group $vL_\mu(\zeta_\tau\mid\tau<\mu)(d)$ is an infinite extension of
$vL_\mu(\zeta_\tau\mid\tau<\mu)$. Take $n\in \N$ and define
\[
\eta_{\mu,n}\>:=\> \zeta_{\mu,n}^{p^{n-1}} -
d_{\mu,s(n)+n}^{p^{n-1}}a_{\mu,s(n)+n}^{1/p}
\>=\>\zeta_{\mu,n-1}^{p^{n-1}}+ \sum_{i=1}^{n-1}
d_{\mu,s(n)+i}^{p^{n-1}}a_{\mu,s(n)+i}^{p^{n-1-i}}\in K\>.
\]
Since $d$ is a pseudo limit of the pseudo Cauchy sequence
$(\zeta_{\mu,n})_{n\in\N}$, we deduce that
\begin{equation}                            \label{d-z}
v (d-\zeta_{\mu,n})\>=\>v z_{\mu,{n+1}}\>=\>
vd_{\mu,s(n+1)+1}a_{\mu,s(n+1)+1}^{1/p} \> >\>
vd_{\mu,s(n)+n}a_{\mu,s(n)+n}^{p^{-n}}\>.
\end{equation}
Therefore,
\begin{eqnarray*}
v\left(d^{p^{n-1}}-\eta_{\mu,n}\right) & = &
v\left(d^{p^{n-1}}- \zeta_{\mu,n}^{p^{n-1}} +
d_{\mu,s(n)+n}^{p^{n-1}}a_{\mu,s(n)+n}^{1/p}\right)\\
& = & p^{n-1} v\left(d-\zeta_{\mu,n}+
d_{\mu,s(n)+n}a_{\mu,s(n)+n}^{p^{-n}}\right)\\
& = & p^{n-1}\min\left\{v(d-\zeta_{\mu,n})\,,\,v\left(
d_{\mu,s(n)+n}a_{\mu,s(n)+n}^{p^{-n}}\right)\right\}\\
& = & p^{n-1}v\left(d_{\mu,s(n)+n}a_{\mu,s(n)+n}^{p^{-n}}
\right) \\
&=& p^{n-1} vd_{\mu,s(n)+n}+\frac{1}{p} va_{\mu,s(n)+n}\>,
\end{eqnarray*}
which shows that
\[
\frac{1}{p} va_{\mu,s(n)+n}\>\in\>vL_\mu(\zeta_\mu\mid\mu<\tau)(d)
\]
for all $n\in\N$. In view of (\ref{vgrsLmu}), these values are not in
$vL_\mu\,$. Since the extension (\ref{extib}) is immediate for $L_\mu$
in place of $L'$, they are also not in $vL_\mu(\zeta_\tau\mid\tau
<\mu)$. It follows that the index $(vL_\mu(\zeta_\tau\mid\tau<\mu)(d):
vL_\mu(\zeta_\tau\mid\tau <\mu))$ is infinite. This contradicts the fact
that the extension $L_\mu(\zeta_\tau\mid\tau<\mu)(d)|L_\mu(\zeta_\tau\mid
\tau<\mu)$ is finite. This contradiction proves that the pseudo Cauchy
sequence $(\zeta_{\mu,n})_{n\in\N}$ is of transcendental type over
$L'(\zeta_\tau \mid\tau<\mu)$. From Theorem~\ref{KT2} it follows
that $(L'(\zeta_\tau\mid\tau\leq \mu)|L'(\zeta_\tau\mid\tau<\mu),v)$ is
an immediate transcendental extension. Since the extension (\ref{extib})
is immediate, we obtain that also $(L'(\zeta_\tau\mid\tau\leq\mu)|L',
v)$ is immediate.

This completes our induction step. By induction on $\mu$ we have
therefore shown that $(L(\zeta_{\tau}\mid\tau<\mu),v)$ is an immediate
extension of $(L,v)$ for each $\mu<\kappa$, which yields that also the
union $(L(\zeta_{\tau}\mid\tau<\kappa),v)$ of these fields is an
immediate extension of $(L,v)$. As every extension $L(\zeta_{\tau}\mid
\tau\leq\mu)|L(\zeta_{\tau}\mid\tau<\mu)$ is transcendental, the
transcendence degree of $L(\zeta_{\tau}\mid\tau<\kappa)$ over $L$ is
at least $\kappa$.

\parb
A simple modification of the above arguments allows us to show the
assertion of part c) of the theorem in the case of $\kappa=[Kv:(Kv)^p]$.
We take the partition of $B$ as in the proof of part b).
We now list the modifications.

Since the $vb=0$ for all $b\in B$, the only requirement for the elements
$d_{\tau,i}$ that we need is that $vd_{\tau,i}<vd_{\tau,j}$ for $i<j$.
If the cofinality of $vK$ is countable, then the elements $d_{\tau,i}$
can be chosen in such a way that the sequence of their values is cofinal
in $vK$. We set
\[
z_{\tau,m}\>:=\> \sum_{i=1}^m d_{\tau,s(m)+i}
b_{\tau,s(m)+i}^{p^{-i}}\in K^{1/p^{\infty}}\>,
\]
Equation (\ref{ztm}) is replaced by
\begin{equation}                            \label{ztm'}
[K(z_{\tau,m}):K] \>=\> p^m \;\;\mbox{ and }\;\;
(b_{\tau,s(m)+1}v)^{1/p},\ldots,(b_{\tau,s(m)+m}v)^{1/p}
\in K(z_{\tau,m})v\>.
\end{equation}
One proves in a similar way as before that
\begin{equation}                            \label{vgrsLmu'}
vL_\mu\>=\>vK \;\;\mbox{ and }\;\; L_\mu v\>=\>Kv((bv)^{1/p}\mid b\in
{\cal B}_\mu)\>.
\end{equation}
In particular, this shows that
\begin{equation}                            \label{vLLv'}
vL=vK \;\;\mbox{ and }\;\; Lv\>=\>(Kv)^{1/p}\>.
\end{equation}
Now the only further part of the proof that needs to be modified is the
one that shows that the extension
$L_\mu(\zeta_\tau\mid\tau<\mu)(d)|L_\mu(\zeta_\tau\mid\tau <\mu)$
cannot be finite. We define $\eta_{\mu,n}$ as before, with ``$b$'' in
place of ``$a$''. Also (\ref{d-z}) holds with ``$b$'' in place of
``$a$'', whence
\[
v\,d_{\mu,s(n)+n}^{-p^{n-1}}\left(d^{p^{n-1}}-
\zeta_{\mu,n}^{p^{n-1}}\right) \>=\> p^{n-1}
vd_{\mu,s(n)+n}^{-1}\left(d-\zeta_{\mu,n}\right)\> >\>0\>.
\]
This leads to
\begin{eqnarray*}
d_{\mu,s(n)+n}^{-p^{n-1}}\left(d^{p^{n-1}}-\eta_{\mu,n}\right)v
& = & d_{\mu,s(n)+n}^{-p^{n-1}}\left(d^{p^{n-1}}-
\zeta_{\mu,n}^{p^{n-1}}+d_{\mu,s(n)+n}^{p^{n-1}}
b_{\mu,s(n)+n}^{1/p}\right)v\\
& = & \left(d_{\mu,s(n)+n}^{-p^{n-1}}(d^{p^{n-1}}-
\zeta_{\mu,n}^{p^{n-1}})+ b_{\mu,s(n)+n}^{1/p}\right)v\\
&=&(b_{\mu,s(n)+n}^{1/p})v\>=\>(b_{\mu,s(n)+n}v)^{1/p}\>,
\end{eqnarray*}
which shows that
\[
(b_{\mu,s(n)+n}v)^{1/p}\>\in\>L_\mu(\zeta_\tau\mid\tau<\mu)(d)v
\]
for all $n\in\N$. In view of (\ref{vgrsLmu'}), these residues are not in
$L_\mu v$. As before, this is shown to contradict $d$ being algebraic
over $L_\mu(\zeta_\tau\mid\tau<\mu)$. This completes our modification
and thereby the proof that $(M|L,v)$ is of transcendence degree at least
$\kappa$.                                             \QED

\parb
We now come to the proof of
\sn
{\bf Proof of Theorem~\ref{distinct_max}}:
\sn
Note that a field $(K,v)$ which satisfies the assumptions of
Theorem~\ref{distinct_max} also satisfies the assumptions of
Theorem~\ref{infpdeg}. We choose the sets $A,B\subseteq K^{1/p}$ and
define $L:=L_{\infty}$ as in the proof of part b) of
Theorem~\ref{infpdeg}. Then, as we have already seen, $(K^{1/p},v)$ is a
maximal immediate extension of $(L,v)$.

\pars
To show the existence of an immediate extension of $L$ of infinite
transcendence degree over $L$, we consider separately the cases i) and
ii) of the theorem. We assume first that the conditions of case i) hold.
Then the set $A$ can be chosen so as to contain an infinite countable
subset $A'$ such that the set of values $S=\{va\mid a\in A'\}$ is
bounded. It must contain a bounded infinite strictly increasing or a
bounded infinite strictly decreasing sequence. If it does not contain
the former, we replace $A'$ by $\{a^{-1}\mid a\in A'\}$, thereby passing
from $S$ to $-S$. Note that in our proof we will not need that
$A'\subseteq A$; we will only use that $A' \subseteq L$. Now we can
choose a sequence $(a_j)_{j\in\N}$ of elements in $A'$ such that the
sequence $(va_j)_{j\in\N}$ is strictly increasing and bounded
by some $\gamma\in vK$. We partition the sequence $(a_j)_{j\in\N}$ into
countably many subsequences
\[
(a_{N,i})_{i\in\N}\qquad (N\in\N)\;.
\]
As in the proof of Theorem~\ref{infpdeg}, we define $K_N:=K(a_{n,i}\mid
n<N\,,\,i\in\N)\subseteq K^{1/p}$.

For every $N\in \N$ we consider the pseudo Cauchy sequence
$(\xi_{N,m})_{m\in\N}$ defined by
\[
\xi_{N,m}\>:=\>\sum_{i=1}^m a_{N,i}^{1/p}\>\in\>K_{N+1}\>.
\]
and the pseudo limits $\xi_N$ of the sequences in the maximal immediate
extension $(K^{1/p}, v)$ of $(L,v)$. We show that for every $N$ the
pseudo limit $\xi_N$ does not lie in the completion $L^c$ of $(L,v)$.
Fix $N\in\N$ and take any $d\in L$. Then $d$ lies already in some finite
extension
\[
E\>:=\>K(a_1^{1/p},\ldots,a_k^{1/p},b_1^{1/p},\ldots,b_l^{1/p})
\]
of $K$ in $L$. Choosing $\xi_E$ as in the proof of
Theorem~\ref{infpdeg}, we obtain that $\xi_E-d\in E$. But from
equalities (\ref{vK'}) with $\mu =0$ it follows that
$v(\xi_N-\xi_E)=\frac{1}{p}va_{N,n}\notin vE$. Thus,
\[
v(\xi_N-d)\>=\> \min\{v(\xi_N-\xi_E),v(\xi_E-d)\}\leq
\frac{1}{p}\,va_{N,n} <\frac{1}{p}\,\gamma\>.
\]
Hence the values $v(\xi_N-d)$, $d\in L$, are bounded by
$\frac{1}{p}\gamma$ and consequently, $\xi_N\notin L^c$.

Again from the proof of Theorem~\ref{infpdeg} it follows that for
every field $K'$ such that $K_1\subseteq K'\subseteq L$ the extension
$(K'(\xi_1)|K',v)$ is immediate and purely inseparable of degree $p$.
Since $\xi_1\notin L^c$, from Proposition~\ref{deformthm} we deduce that
for an element $d_1\in K^{\times}$ satisfying inequality (\ref{dist})
with $\eta=\xi_1$, a root $\vartheta_1$ of the polynomial
\[
f_1\>:=\>X^p-X-\left(\frac{\xi_1}{d_1}\right)^p
\]
generates an immediate Galois extension $(L(\vartheta _1)|L,v)$ of
degree $p$ with a unique extension of the valuation $v$ from $L$ to
$L(\vartheta _1)$. Take any field $K'$ such that $K_1\subseteq
K'\subseteq L$. Then $\xi_1\notin K'^c$ and the element $d_1$ satisfies
inequality (\ref{dist}) with every element $c\in K'$. Therefore also
$(K'(\vartheta_1)|K',v)$ is an immediate extension of degree $p$ with
a unique extension of $v$ from $K'$ to $K'(\vartheta_1)$.

Take any $m>1$. Suppose that we have shown that for every $l< m$ there is
$d_l\in K ^{\times}$ such that a root $\vartheta_l$ of the polynomial
\[
f_l\>:=\>X^p-X-\left(\frac{\xi_l}{d_l}\right)^p
\]
generates, for any field $K'$ with $K_{l+1}\subseteq K'\subseteq L$, an
immediate Galois extension $(K'(\vartheta_1,\ldots, \vartheta_l)|
K'(\vartheta_1,\ldots,\vartheta_{l-1}),v)$ of degree $p$ with a unique
extension of the valuation $v$ from $K'(\vartheta_1, \ldots,
\vartheta_{l-1})$ to $K'(\vartheta_1, \ldots,\vartheta_{l})$. Then in
particular, the extension $(K_{m+1}(\vartheta_1,\ldots,\vartheta_{m-1})|
K_{m+1},v)$ is immediate. Take any field $K'$ such that $K_{m+1}
\subseteq K'\subseteq L$. Replacing in the argumentation of the proof of
part b) of Theorem~\ref{infpdeg} the field $K'(\xi_l\mid l< m)$ by
$K'(\vartheta_l\mid l< m)$, we deduce that $(K'(\vartheta _1,\ldots,
\vartheta_{m-1})(\xi_m)| K'(\vartheta _1,\ldots,\vartheta_{m-1}),v)$ is
an immediate purely inseparable extension of degree $p$. Since
$\xi_m\notin L^c$ and $L(\vartheta _1,\ldots,\vartheta_{m-1})^c
=L^c(\vartheta_1,\ldots, \vartheta_{m-1})$ is a separable extension of
$L$, linearly disjoint from the purely inseparable extension
$L^c(\xi_m)|L^c$, we obtain that
\[
[L^c(\vartheta _1,\ldots,\vartheta_{m-1})(\xi_m):
L^c(\vartheta _1,\ldots,\vartheta_{m-1})]\>=\> p.
\]
Therefore, $\xi_m$ does not lie in $L(\vartheta _1,\ldots,
\vartheta_{m-1})^c$. Thus, from Proposition~\ref{deformthm} it follows
that for an element $d_m\in K^{\times}$ satisfying inequality
(\ref{dist}) with $\eta=\xi_m$, a root $\vartheta_m$ of the polynomial
$f_m:=X^p-X-\left(\xi_m/d_m\right)^p$
generates an immediate Galois extension $(L(\vartheta_1,\ldots,
\vartheta_m)| L(\vartheta _1,\ldots,\vartheta_{m-1}),v)$ of degree $p$
with a unique extension of the valuation $v$ from $L(\vartheta_1,\ldots,
\vartheta_{m-1})$ to $L(\vartheta_1,\ldots,\vartheta_m)$. As in the
case of $m=1$ we deduce that also the extension $(K'(\vartheta_1,
\ldots,\vartheta_{m})| K'(\vartheta_1,\ldots, \vartheta_{m-1}),v)$ is
immediate and the valuation $v$ of $K'(\vartheta _1,\ldots,
\vartheta_{m-1})$ extends uniquely to $K'(\vartheta _1,\ldots,
\vartheta_{m})$.

By induction, we obtain an infinite immediate separable-algebraic
extension $F:=L(\vartheta_{m}\mid m\in\N)$ of $L$ with the unique
extension of the valuation $v$ of $L$ to $F$. Thus, the extension $F|L$
is linearly disjoint from $L^h|L$. From the separable algebraic case of
Theorem~\ref{MTai} it follows that each maximal immediate extension
$(M,v)$ of $(F,v)$ has infinite transcendence degree over $F$. Since
$(F|L,v)$ is immediate, $M$ is also a maximal immediate extension of
$L$.

\parm
Similar arguments allow us to prove the assertion in the case of an
infinite residue field extension $Kv|(Kv)^p$ when the value group $vK$
is not discrete. Let us describe the modifications.

Take an infinite countable subset $B'$ of $B$ and an infinite partition
of $B$ into infinite sets
\[
B_N\>=\>\{ b_{N,i}\mid i\in\N\}\qquad (N\in\N)\;.
\]
Since $vK$ is not discrete, we can choose elements $c_i\in K$ such that
the sequence $(vc_i)_{i\in\N}$ of their values is strictly increasing
and bounded by some element $\gamma\in vK$. For every $N$ we consider
the pseudo Cauchy sequence $(\xi_{N,m})_{m\in\N}$ defined by
(\ref{xi_resf}).

The only further part of the proof that needs to be modified is the one
that shows that $\xi_N\notin L^c$. More precisely, we need to show that
for any element $d\in E$ we have that $v(\xi_N-d)<\gamma$.
Take $\xi_E$ as in the second case of the proof of part b) of
Theorem~\ref{infpdeg}. From the equalities (\ref{vK'}) with $\mu=0$ we
deduce that $c^{-1}_n (\xi_N-\xi_E)v=(b_{N,n}v)^{1/p}\notin Ev$.
Suppose that $v(\xi_N-d)> v(\xi_n-\xi_E)$. Then
\[
v\left(c_n^{-1}(\xi_N-\xi_E) - c_n^{-1}(\xi_E-d)\right)\> >\>
v c^{-1}_n(\xi_N-\xi_E)\>=\>0.
\]
It follows that $c_n^{-1}(\xi_N-\xi_E)v=c_n^{-1}(\xi_E-d)v\in Ev$,
a contradiction. Consequently,
\[
v(\xi_N-d)\> \leq \> v(\xi_N-\xi_E)\> =\> vc_n\> < \> \gamma\>.
\]
This completes our modification and thereby the proof that $(L,v)$ admits
a maximal immediate extension of infinite transcendence degree over $L$.
\QED

\parb
Finally, we give the
\sn
{\bf Proof of Theorem~\ref{maxtrext}}:
\sn
Take an extension $(L|K,v)$ as in the assumptions of the theorem.
In view of the value-algebraic and residue-algebraic cases of
Theorem~\ref{MTai}, it suffices to show that at least one of the
extensions $vL|vK$ or $Lv|Kv$ is infinite.

Take $K'$ to be the relative algebraic closure of $K$ in $L^h$.
By the assumptions on the residue field and value group extensions of
$(L|K,v)$, it follows from Lemma~\ref{vrrac} that $vK'=vL^h=vL$ and
$K'v=L^hv=Lv$. Therefore, $(L^h|K',v)$ is an immediate transcendental
extension.

Suppose that the value group extension and the residue field extension
of $(L|K,v)$ and hence of $(K'|K,v)$ were finite. But since $K$ is
henselian and a defectless field by Theorem~\ref{mfhdl}, the degree
$[K':K]$ is equal to $(vK':vK)[K'v:Kv]$ and hence would be finite, so
$(K',v)$ would again be a maximal field, which contradicts the fact that
$(L^h|K',v)$ is a nontrivial immediate extension.                \QED

\bn

\newcommand{\lit}[1]{\bibitem{#1}}

\fvkadresse
\abaddress
\end{document}